\documentclass[a4paper,11pt]{article}

\usepackage{float}
\usepackage{amsmath}
\usepackage{amsthm}
\usepackage{amsfonts}
\usepackage{amssymb}
\usepackage[T1]{fontenc}
\usepackage{geometry}
\usepackage{graphicx}
\usepackage{hyperref}
\usepackage{icomma}
\usepackage[latin1]{inputenc}
\usepackage{latexsym}
\usepackage[numbers]{natbib}
\usepackage{textcomp}
\usepackage[all]{xy}
\usepackage{cancel}
\usepackage{color}
\usepackage{dsfont}

\title{General approximation method for the distribution of Markov processes conditioned not to be killed}
\author{Denis Villemonais
\thanks{Institut \'Elie Cartan de Nancy, Universit\'e de Lorraine; TOSCA project-team, INRIA Nancy -- Grand Est; IECN -- UMR 7502, Universit\'e de Lorraine,
B.P. 70239, 54506 Vandoeuvre-l\`es-Nancy Cedex, France }
}

\DeclareMathSymbol{\minus}{\mathord}{operators}{"2D}

\newtheorem{theorem}{Theorem}

\newtheorem{proposition}[theorem]{Proposition}

\newtheorem{hypothesis}{Hypothesis}

\def\be{\begin{eqnarray}}
\def\ee{\end{eqnarray}}
\def\ben{\begin{eqnarray*}}
\def\een{\end{eqnarray*}}
\def\bei{\begin{itemize}}
\def\eei{\end{itemize}}
\def\me{\medskip \noindent}
\def\bi{\bigskip \noindent}
\def\E{\mathbb{E}}
\def\P{\mathbb{P}}
\def\R{\mathbb{R}}
\def\1{\mathbf{1}}
\def\N{\mathbb{N}}

\def\d{\partial}

\begin{document}
\maketitle

\begin{abstract}
We consider a strong Markov process with killing and prove an approximation method for the distribution of the process conditioned not to be killed when it is observed. The method is based on a Fleming-Viot
  type particle system with rebirths, whose particles evolve as
  independent copies of the original strong Markov process and jump
  onto each others instead of being killed. Our only assumption is that the
  number of rebirths of the Fleming-Viot type system doesn't explode in
  finite time almost surely and that the survival probability of the original process remains positive in finite time. The approximation method generalizes previous results and comes with a speed of convergence.

  A criterion for the non-explosion of the number of rebirths is also provided
  for general systems of time and environment dependent diffusion
  particles. This includes, but is not limited to, the case of the Fleming-Viot type
  system of the approximation method. The proof of the non-explosion criterion uses
  an original non-attainability of $(0,0)$ result for pair of
  non-negative semi-martingales with positive jumps.
\end{abstract}

\section{Introduction}

Markov processes with killing are Markov processes which stop to evolve after a random time, called the \textit{killing time}. 
While the behavior of a Markov process after this time is trivial, its distribution before its killing time represents a substantial information and thus several studies
concentrate on the distribution of a process conditioned not to be
killed when it is observed (we refer the reader to the extensive bibliography updated by Pollett \cite{Pollett}, where several studies on models with killing mechanisms are listed). The main motivation of the present paper is to provide a general approximation method for this distribution.

\me
There
are mainly two ways of \textit{killing} a Markov process, both of them being handled in the present paper. The first kind of killing occurs when the
process reaches a given set. For instance, a demographic
model is stopped when the population becomes extinct, that is when the size of the population hits $0$.
The second kind of
killing occurs after an exponential time. For example, a
chemical particle usually disappears by reacting with another one
after an exponential time, whose rate depends on the concentration of
reactant in the medium. If the killing time is given
by the time at which the process reaches a set, we call it a
\textit{hard} killing time. If it is given by an exponential clock, we
call it a \textit{soft} killing time.  
 In order to formally describe a process with killing in a general setting, we consider a continuous time strong Markov process $\cal Z$ evolving in a measurable topological state space $F\cup\{\d\}$, where $\d\notin F$ is absorbing. This means that ${\cal Z}$ evolves in F until it reaches $\d$ and then remains in $\d$ forever. The \textit{killing time} of $\cal Z$ is defined as the hitting time of $\d$ and is denoted by $\tau_{\d}=\inf\{t\geq 0,{\cal Z}_t=\d\}$. This general setting can be used to model both \textit{hard} and \textit{soft} killing types.

\me
The main difficulty in approximating the distribution of a Markov process before its killing is that the probability of the event \textit{"the process is still not killed at time $t$"} vanishes when $t$ goes to infinity. Indeed, this is a rare event simulation problem and, in particular, naive Monte-Carlo methods are not well-suited (the number of \text{still not killed} trajectories vanishing with time, a Monte-Carlo algorithm would undergo an unavoidable discrepancy in the long term). Our approximation method is based on a quite natural modification of the Monte-Carlo method, where killed particles are reintroduced through a rebirth in $F$ (the state space of the still not killed particles) in order to keep a constant number of meaningful particles through time. This modification, which appears to be a Fleming-Viot type particle system, has been introduced by Burdzy, Holyst,
Ingermann and March in~\cite{Burdzy1996} and is described in the following definition.

\me
\begin{description}
\item[Definition.] 
Let $\cal Z$ be a strong Markov process with killing evolving in $F\cup\{\d\}$. Given $N\geq 2$ and $(x_1,\ldots,x_N)\in F^N$, we define the \textit{Fleming-Viot type particle system with $N$ particles and evolving as $\cal Z$ between their rebirths}, starting from $(x_1,\ldots,x_N)$ and denoted by $FV^N_{\cal Z}(x_1,\ldots,x_N)$, as follows.
The particles of the
 system start from $(x_1,\ldots,x_N)$ at time $0$ and evolve
 as $N$ independent copies of ${\cal Z}$ until one of
 them is killed. Then
 \begin{itemize}
 \item[--] if two or more particles are simultaneously killed at this moment, or if a particle jumps simultaneously without being killed (so that a particle is killed and another one jumps), then we stop the definition of the particle system and say that the process undergoes a \textit{failure}.
 \item[--] otherwise (\textit{i.e.} if one and only one particle is killed and the other particles do not jump at this killing time) then the unique killed particle is taken from the absorbing point $\d$ and is instantaneously placed at the position of a particle chosen uniformly between the $N-1$ remaining ones; in this situation we say that the particle undergoes a \textit{rebirth}.
\end{itemize}
 After this operation and in the case of a rebirth, each of the $N$ particles lies in $F$. They then evolve
 as $N$ independent copies of ${\cal Z}$, until one of
 them is killed and so on. This procedure defines the particle system $FV^N_{\cal Z}(x_1,\ldots,x_N)$ in an incremental way, rebirth after rebirth and until it undergoes a failure.
\end{description}

\me
One particularity of this method is that the number of rebirths faces a risk of explosion in finite time: in some cases, that is for some choices of $\cal Z$ and $N\geq 2$, the number of rebirths in the particle system explodes to $+\infty$ in finite time. In such a case, the particle system is well defined up to the explosion time, but there is no natural way to extend its definition after this time (we refer the reader to~\cite{Bieniek2011} for a non-trivial case where the number of rebirths explodes in finite time with positive probability).
 This leads us to the following assumption (see also Section~\ref{section:non-explosion}, where we provide a sufficient condition for a diffusion process to satisfy  Hypothesis A($N$)).
 
\begin{description}
\item[Hypothesis A($N$), $N\geq 2$.]
 A Markov process with killing $\cal Z$ is said to fulfil Hypothesis~A($N$) if and only if 
the \textit{Fleming-Viot type particle system with $N$ particles evolving as $\cal Z$ between their rebirths} undergoes no failure and if its number of rebirths remains finite in finite time almost surely, for any initial distribution of the particle system.
\end{description}

\me In Section~\ref{section:approximation}, we fix $T>0$ and consider a continuous time Markov process with killing $\cal Z$. We state and prove an approximation method for the distribution of ${\cal Z}_T$ conditionally to $T<\tau_{\d}$. This method, based on Fleming-Viot type particle systems, does not require that $\cal Z$ satisfies Hypothesis~A($N$) for some $N\geq 2$. Instead, we assume that there exists a sequence $({\cal Z}^N)_{N\geq 2}$ of Markov processes such that ${\cal Z}^N_T$ converges in law to ${\cal Z}_T$ when $N\rightarrow\infty$ and such that ${\cal Z}^N$ satisfies Hypothesis~A($N$), for all $N\geq 2$. 
In particular
the Fleming-Viot type system with $N$ particles built over ${\cal Z}^N$ is well defined for any $N\geq 2$ and any initial distribution. This particle system, denoted by $FV_{{\cal Z}^N}^N$, evolves in $F^N$ and we denote by $\mu^N_t$ its empirical distribution at time $t\geq0$. Our main result states that, if the sequence of initial empirical distributions $(\mu^N_0)_{N\geq 2}$ converges to a probability measure $\mu_0$ on $F$, then the sequence of random probability measures $(\mu^N_T)_{N\geq 0}$ converges in law to the distribution of ${\cal Z}_T$, initially distributed as $\mu_0$ and conditioned to $T<\tau_{\d}$. Also, we provide a speed of convergence for the approximation method. This result generalises the approximation method
proved by Grigorescu and Kang in~\cite{Grigorescu2004} for standard
multi-dimensional Brownian motion, by Villemonais~\cite{Villemonais2010} for
Brownian motions with drift and by Del Moral and Miclo for softly
killed Markov processes (see~\cite{DelMoral2003} and references
therein) and by Ferrari and Mari\`c
\cite{Ferrari2007}, which regards continuous time Markov chains in
discrete spaces.

\me Our result gives a new insight on this approximation method through three key points of generalisation. Firstly, this new result allows both hard and soft killings, which is a
natural setting in applications: typically, species can disappear
because of a lack of newborns (which corresponds to a
hard killing at $0$) or because of a brutal natural catastrophe (which
typically happens following an exponential time). Secondly, we
implicitly allow time and environment dependencies, which is quite
natural in applications where individual paths can be influenced by
external stochastic factors (as the changing weather).  Last but not least, for any $N\geq 2$, we do not require that $\cal Z$ satisfies Hypothesis~A($N$). Indeed,
we only require that there exists an approximating sequence $({\cal Z}^N)_{N\geq 2}$ such that ${\cal Z}^N$ fulfils Hypothesis~A($N$), $\forall N\geq 2$.  This is of first importance, since the non-explosion of the number of rebirths is a quite hard problem which remains open in several situations. It is for instance the case for  diffusion processes with unbounded drift, for
diffusions killed at the boundary of a non-regular domain
and for Markov processes with unbounded rates
of killing. Thus in our case, the three irregularities can be handled by
successive approximations of the coefficients, domain and rate of
killing. For these reasons, the approximation result proved in the present paper can be used as a general approximation method for the distribution of conditioned Markov processes.

\bi The approximation method being proved in this very general setting,  the only remaining difficulty is to provide a sufficient criterion for Hypothesis~$A(N)$ to hold, for any $N\geq 2$.
In Section~\ref{section:non-explosion}, we consider a multi-dimensional diffusion process $\cal Z$ and provide a sufficient criterion for Hypothesis~A($N$) to be fulfilled for $FV_{\cal Z}^N$, the Fleming-Viot type system with $N$ particles evolving as $\cal Z$ between their rebirths. We allow both soft and hard killings and we allow the coefficients of $\cal Z$ to be time-inhomogeneous and environment dependent.
 This criterion generalises similar non-explosion and non-failure results recently proved by L\"obus in~\cite{Lobus2009} and by Bienek,
Burdzy and Finch in~\cite{Bieniek2009} for Brownian particles killed
at the boundary of an open set of $\mathbb{R}^d$, by Grigorescu and Kang in~\cite{Grigorescu2011} for time-homogeneous particles driven by a
stochastic equation with regular coefficients killed at the boundary
of an open set and by Villemonais in~\cite{Villemonais2010} for Brownian particles with drift. 
Other models
of diffusions with rebirths from a boundary have also been introduced in
\cite{Ben-Ari2009}, with a continuity condition on the jump measure
that isn't fulfilled in our case, in~\cite{Grigorescu2007}, where fine
properties of a Brownian motion with rebirth have been established,
and in \cite{Kolb2011a}, \cite{Kolb2011}, where Kolb and W\"ukber have
studied the spectral properties of similar models. 

\me Our non-explosion result is a generalization of the previously cited ones and is actually not restricted to Fleming-Viot type particle systems. The first generalization axis concerns the Markov processes that drive the particles between their rebirths: our criterion allows  a different state space and a different dynamic for each particle. Moreover these dynamics can be time-inhomogeneous environment-dependent diffusion processes, while previous results are restricted to time-homogeneous diffusion processes. The second aspect of the generalization is related to the rebirths mechanisms: we allow the rebirth position of a particle to be chosen with a large degree of freedom (in particular, not necessarily on the position of another particle) and the
whole particle system is allowed to undergo a rebirth when a particle is killed. This setting includes (but is not limited to) the Fleming-Viot type particle system case, so that our non-explosion result validates the
approximation method described above for time/environment dependent
diffusions with hard and soft killing. For non Fleming-Viot type rebirths mechanisms, it remains an open problem to determine the existence and, if appropriate, the value of the limiting empirical distribution.

\me The proof of the non-explosion is partly based on an original
non-attainability of $(0,0)$ result for semi-martingales, which is
stated in Section~\ref{section:non-attainability} of this paper.
Note that this result  answers a different problematic from the rest of the paper and has its own interest. Indeed, inspired by Delarue~\cite{Delarue2008}, our non-attainability criterion generalizes existing non-attainability criteria by considering general semi-martingales (\textit{i.e.} not necessarily obtained as solutions of stochastic differential equations). This essential improvement is a \textit{sine qua non} condition for the development of the non-explosion criterion of Section 3.

\section{Approximation of a Markov process conditioned not to be killed}
\label{section:approximation}
\label{SeApproximation}
We consider a c\`adl\`ag strong Markov process ${\cal Z}$ evolving in a state space $F\cup\{\d\}$, where $\d\notin F$ is an \textit{absorbing point} for $\cal Z$, also called the \textit{cemetary point} for $\cal Z$ (this means that the process remains in $\{\d\}$ forever after reaching it). When $\cal Z$ reaches $\d$, we say that it is \textit{killed} and we set $\tau_{\d}=\inf\{t\geq 0,\ {\cal
  Z}_t=\partial\}$.
 In this section, we assume that we are given a random probability measure $\mu_0$ on $F$ and a deterministic time $T\geq 0$.
 We        are interested in the approximation of 
$\P_{\mu_0}\left({\cal Z}_T\in\cdot\mid T<\tau_{\d}\right)$, \textit{i.e.} 
of  the distribution of ${\cal Z}_T$ with initial distribution $\mu_0$ and conditioned to $T<\tau_{\d}$.

\me Let $(\mu^N_0)_{N\geq 2}$ be a sequence of empirical distributions with $N$ points which converges weakly to $\mu_0$ (by an \textit{empirical distribution with $N$ points} $\mu^N_0$, we mean that there exists a random vector $(x_1,\ldots,x_N)$ in $F^N$ such that $\mu^N_0$ has the same law as $\frac{1}{N}\sum_{i=1}^N\delta_{x_i}$).
We assume that there exists a sequence $({\cal Z}^N)_{N\geq 2}$ of strong Markov processes with killing evolving in $F\cup\{\d\}$, such that ${\cal Z}^N$ fulfils Hypothesis~A($N$) for any $N\geq 2$ and such that, for any bounded continuous function $f:F\rightarrow \mathbb{R}$,
\begin{align}
\label{eq:for-remark1}
    \E_{\mu_0^N}\left(f({\cal Z}^N_T)\1_{T<\tau^{N}_{\d}}\right)\xrightarrow[N\rightarrow\infty]{law} \E_{\mu_0}\left(f({\cal Z}_T)\1_{T<\tau_{\d}}\right),
\end{align}
where $\tau^N_{\d}=\inf\{t\geq 0,\ {\cal Z}^N_t=\d\}$ denotes the killing time of ${\cal Z}^N$.

\bi\textit{Remark 1.} 
A particular case where \eqref{eq:for-remark1} holds true (we still assume that $\mu_0$ is a random measure) is when $\mu_0^N$ is the empirical distribution of a vector distributed with respect to $\mu_0^{\otimes N}$ for any $N\geq 2$ and when, for all $x\in F$
and all continuous and bounded function $f:F\mapsto\mathds{R}_+$,
\begin{align}
  \label{EqReHyThUniform}
  P^N_T f(x) \xrightarrow[N\rightarrow\infty]{} P_T f(x).
\end{align}
Indeed, we have
\begin{equation*}
\mu^{N}_0\left(P^N_T f\right)
  \stackrel{law}{=} \frac{1}{N}\sum_{i=1}^N \left[ P^N_{T} f(x_i) - \mu_0 \left( P^{N}_T f \right)   \right]+ \mu_0\left(P^{N}_T f\right),
\end{equation*}
where $(x_i)_{i\geq 1}$ is an \textit{iid} sequence of random
variables with law $\mu_0$. By the law of large numbers, the first
right term converges to $0$ almost surely and, by~\eqref{EqReHyThUniform} and by dominated convergence, the
second right term converges almost surely to 
$\mu_0\left( P_T f \right)$.

\bi We define now the $N$-particles system $\mathds{X}^{N}=\left(X^{1,N},...,X^{2,N}\right)$ as the \textit{Fleming-Viot type system whose $N$ particles evolve as independent copies of ${\cal Z}^N$ between their rebirths} and with initial empirical distribution $\mu_0^N$ (thus  $\mathds{X}^N=FV^N_{{\cal Z}^N}(x_1,\ldots,x_N)$, where $(x_1,\ldots,x_N)$ is such that $\mu_0=1/N\sum_{i=1}^N\delta_{x_i}$). Since we assumed that ${\cal Z}^N$ fulfils Hypothesis~A($N$), the particle system $\mathds{X}^{N}$ is well defined at any time $t\geq 0$. Also, for any $i\in\{1,\ldots,N\}$, the particle $X^{i,N}$ is a c\`adl\`ag process evolving in $F$. We denote 
by $\mu^{N}_t$ the empirical distribution of $\mathds{X}^{N}_t$ at time $t\geq 0$:
\begin{equation*}
  \mu^{N}_t=\frac{1}{N}\sum_{i=1}^N \delta_{X^{i,N}_t}\ \in {\cal M}_1(F)\text{ almost surely},
\end{equation*}
where ${\cal M}_1(F)$ denotes the set of probability measures on $F$.

\begin{theorem}
  \label{theorem:convergence}
 Assume that the survival probability of $\cal Z$ at time $T$ starting with distribution $\mu_0$ is strictly positive almost surely, which means that $\P_{\mu_0}\left( T<\tau_{\d} \right)>0$ almost surely.
	Then, for any continuous and bounded
  function $f:F\mapsto \mathds{R}_+$,
  \begin{equation*}
    \mu^{N}_T(f)\xrightarrow[N\rightarrow\infty]{law} 
    \E_{\mu_0}\left(f({\cal Z}_T)|T<\tau_{\d}\right).
  \end{equation*}
  Moreover, for any bounded measurable function $f:F\mapsto \mathds{R}_+$, we have
  \begin{equation*}
    E\left(\left|\mu^{N}_T(f)- \E_{\mu_0^N}\left(f({\cal Z}^N_T)|T<\tau^N_{\d}\right)
      \right|\right)
    \leq
    \frac{2(1+\sqrt{2})\|f\|_{\infty}}{\sqrt{N}}\sqrt{E\left(\frac{1}{\left(\P_{\mu^{N}_0}\left(T<\tau^N_\d\right)\right)^2}\right)}.
  \end{equation*}
\end{theorem}

\bi\textit{Remark 2.}  In~\cite{Rousset2006}, Rousset consider a process $\cal Z$ which is only subject to soft killings, with a uniformly
  bounded killing rate. In this particular case, the author proves a uniform rate of convergence over all times $T$, using the stability of
  the underlying Feynman-Kac semi-group. Our result shall be used in further work to extend Rousset's results to many process with sufficiently stable associated Feynman-Kac semi-group.

\smallskip
\noindent
\textit{Remark 3.} We emphasize that our approximation method only concerns the conditional distribution of the process. For results on the pathwise behaviour of the process conditioned not to be killed in the future, we refer the reader to the theory of $Q$-processes, which are processes evolving as $\cal Z$ conditioned to never be killed (see for instance~\cite{Lambert2007}).

\smallskip
\noindent
\textit{Remark 4.} The \textit{quasi-limiting distribution}, \textit{i.e.} $\lim_{t\rightarrow+\infty}\P_{\mu_0}\left({\cal Z}_T\in\cdot\mid T<\tau_{\d}\right)$, is the subject of several studies and is closely related (in fact equivalent) to the concept of \textit{quasi-stationary distribution} (see for instance~\cite{Meleard2011}). Theorem~\ref{theorem:convergence} implies that the quasi-limiting distribution, if it exists, is equal to $\lim_{t\rightarrow+\infty}\lim_{N\rightarrow\infty}\mu^N_t$.  Whether this limit is also equal to $\lim_{N\rightarrow\infty}\lim_{t\rightarrow+\infty}\mu^N_t$ (which has practical interests in numerical computation) remains an open problem in many situations and has been resolved quite recently in some particular cases (see for instance \cite{Burdzy2000},\cite{Ferrari2007},\cite{Grigorescu2011} and \cite{Villemonais2010}).

\begin{proof}[Proof of Theorem \ref{theorem:convergence}]
For any $N\geq 2$ and any $i\in\{1,\ldots,N\}$, we denote the increasing sequence of rebirth times of the particle $X^{i,N}$ by
$$ \tau^{i,N}_1 < \tau^{i,N}_2 < \cdots < \tau^{i,N}_n <\cdots $$
For any $t\geq 0$, we denote by $A^{i,N}_t$ the number of rebirths of the particle $X^{i,N}$ before time $t$:
$$
A^{i,N}_t=\max\{n\in\N,\ \tau^{i,N}_n\leq t\}.
$$
We also set $A^N_t=\sum_{i=1}^N A^{i,N}_t$ as the total number of rebirths in the $N$-particles system $\mathbb{X}^N$ before time $t$. Since ${\cal Z}^N$ is assumed to fulfil Hypothesis~A($N$), we have $\lim_{n\rightarrow\infty} \tau^{i,N}_n=+\infty$ almost surely and $A^N_t<+\infty$ for all $t\geq 0$ almost surely.

\smallskip \noindent For any $N\geq 2$, we denote by $(P^N_t)_{t\geq 0}$ the sub-Markov semi-group associated with the Markov process ${\cal Z}^N$. It is defined, for any bounded measurable function $f$ on $F\cup \{\d\}$ and any $t\geq 0$, by
$$
P^N_t f(x)=\E_x\left(f({\cal Z}^N_t)\1_{t<\tau^N_{\d}}\right),\ \forall x\in F\cup\{\d\}.
$$

\me The proof of Theorem~\ref{theorem:convergence} is divided into three steps. In a first step, we fix $N\geq 2$
and we prove that, for any bounded and measurable function $f$ on $F$, there exists a local martingale
$M^{N}$ such that
\begin{equation}
\label{EqTUsum1}
\mu^{N}_t\left(P^N_{T-t}f\right)=\mu^{N}_0\left(P^N_T f\right) +
M^{N}_t +
\frac{1}{N}\sum_{i=1}^N{\sum_{n=1}^{A^{i,N}_t}{
    \left[ 
      \frac{1}{N-1}\sum_{j\neq i}{P^N_{T-\tau_n^{i,N}}f(
        X^{j,N}_{\tau_n^{i,N}})}
      \right]
}},\ \forall t\in[0,T].
\end{equation}
In a second step, we define the non-negative measure $\nu^{N}_t$ on $F$ by
\begin{equation*}
  \nu^{N}_t(dx)=\left(\frac{N-1}{N}\right)^{A^{N}_t}\mu^{N}_t(dx).
\end{equation*}
The measure $\nu^{N}$ is obtained by multiplying the empirical measure $\mu^N$ by $(N-1)/N$ at each rebirth. This loss of mass is introduced in order to
compensate the last right term in~\eqref{EqTUsum1}: we prove that
$\nu^{N}_t\left(f\right) - \mu^{N}_0\left(P^{N}_t f\right)$ is a local martingale whose $L^2$ norm is bounded by $(1+\sqrt{2})\|f\|_{\infty}/\sqrt{N}$, thus
\begin{align}
\label{eq:sec1-maj-before-normalization}
    \sqrt{E\left(\left|\nu^{N}_{T}(f)-
          \mu^{N}_0(P^N_{T} f)\right|^2\right)}\leq
    \frac{(1+\sqrt{2})\|f\|_{\infty}}{\sqrt{N}}.
\end{align}
 In step $3$, we note that $\nu^{N}_T$ and $\mu^{N}_T$ are almost surely proportional and conclude the proof of Theorem~\ref{theorem:convergence} by renormalizing
 $\nu^{N}_T$ and $\mu^{N}_0\left(P^{N}_T \cdot\right)$ in~\ref{eq:sec1-maj-before-normalization}.

\bi \textit{Step 1: proof of decomposition~\eqref{EqTUsum1}.}\\
 Fix $N\geq 2$ and let
$f:F\cup\{\partial\}\mapsto \mathds{R}_+$ be a measurable bounded
function such that $f(\partial)=0$. 
We define, for all $t\in[0,T]$ and
$z\in F$,
\begin{equation*}
  \psi_t(z)=P^N_{T-t}f(z).
\end{equation*}
The process $\left(\psi_t({\cal Z}^N_t)\right)_{t\in[0,T]}$ is
a martingale with respect to the natural filtration of ${\cal Z}^N$ and is equal to $0$ at time $\tau_{\partial}^N$
on the event $\{\tau_{\partial}^N\leq T\}$.  Indeed,
for all $s,t\geq 0$ such that $s+t\leq T$, the Markov
property implies
\begin{equation*}
  E\left(\psi_{t+s}({\cal Z}^N_{t+s})|\left({\cal
    Z}^N_u\right)_{u\in[0,t]}\right)=P^N_{s}\psi_{t+s}({\cal
    Z}^N_t)=\psi_t({\cal Z}^N_t)
\end{equation*}
and $P^N_t f(\d)=0$ for any $t\geq 0$ yields
\begin{equation*}
  \psi_{\tau^N_{\partial}\wedge T}({\cal
    Z}^N_{\tau^N_{\partial}\wedge
    T})=\psi_{\tau^N_{\partial}}(\partial)\mathds{1}_{\tau^N_{\partial}\leq
    T}+\psi_{T}({\cal
    Z}^N_T)\mathds{1}_{\tau^N_{\partial} > T}=\psi_{T}({\cal
    Z}^N_T)\mathds{1}_{\tau^N_{\partial} > T}.
\end{equation*}

\me Fix $i\in\{1,...,N\}$. For any $n\geq 0$, we define the process
$\left(\mathds{M}^{i,n}_t\right)_{t\in[0,T]}$ -- which implicitely depends on $N$ -- by
\begin{equation*}
  \mathds{M}^{i,n}_t=\mathds{1}_{t<\tau^{i,N}_{n+1}}\psi_{t}(X^{i,N}_{t})-\psi_{t\wedge
    \tau^{i,N}_{n}}(X^{i,N}_{t\wedge\tau^{i,N}_{n}}),
\  \text{with } \tau^{i,N}_0:=0.
\end{equation*}
Since $X^{i,N}$ evolves as ${\cal Z}^N$ in the time interval
$[\tau^{i,N}_n,\tau^{i,N}_{n+1}[$, $\mathds{M}^{i,n}_t$ is a martingale with respect to the natural filtration of the whole particle system
and
\begin{equation*}
  \mathds{M}^{i,n}_t=
  \left\lbrace
  \begin{array}{l}
\displaystyle  0,\text{ if } t<\tau^{i,N}_{n}\ \text{ \textit{i.e.} }  A^{i,N}_t<n,\\
\displaystyle     \psi_{t}(X^{i,N}_{t})-\psi_{\tau^{i,N}_{n}}(X^{i,N}_{\tau^{i,N}_n}),\text{ if } t\in[\tau^{i,N}_{n},\tau^{i,N}_{n+1}[\ \textit{ i.e. } A^{i,N}_t=n,\\
\displaystyle     -\psi_{\tau^{i,N}_n}(X^{i,N}_{\tau^{i,N}_n}),\text{ if } t>\tau^{i,N}_{n+1}\ \textit{ i.e. } n < A^{i,N}_t.
  \end{array}
    \right.    
\end{equation*}
Summing over all rebirth times up to time $t$, we get
\begin{equation}
  \label{EqTUmartingale1}
  \psi_t(X^{i,N}_t)=\psi_0(X^{i,N}_0)+
  \sum_{n=0}^{A^{i,N}_t}\mathds{M}^{i,n}_t+\sum_{n=1}^{A^{i,N}_t}{\psi_{\tau_n^{i,N}}(
    X^{i,N}_{\tau_n^{i,N}})}.
\end{equation}
Defining the local martingales
\begin{equation*}
  \mathds{M}^{i}_t=\sum_{n=0}^{A^{i,N}_t} \mathds{M}^{i,n}_t \text{ and }     \mathds{M}_t=\frac{1}{N}\sum_{i=1}^N \mathds{M}^{i}_t
\end{equation*}
and summing over $i\in\{1,...,N\}$, we obtain
\begin{equation*}
  \mu_t^{N}(\psi_t)=\mu^{N}_0(\psi_0)+\mathds{M}_t+\frac{1}{N}\sum_{i=1}^N \sum_{n=1}^{A^{i,N}_t} {\psi_{\tau_n^{i,N}}(
    X^{i,N}_{\tau_n^{i,N}})}.
\end{equation*}
  At each rebirth time $\tau_n^{i,N}$, the rebirth position of the particle
  $X^{i,N}$ is chosen with respect to the empirical
  measure of the other particles.
  The expectation of ${\psi_{\tau_n^{i,N}}(
    X^{i,N}_{\tau_n^{i,N}})}$ conditionally to the position of the
  other particles at the rebirth time is then the average value $
  \frac{1}{N-1}\sum_{j\neq i} {\psi_{\tau_{n-}^{i,N}}(
    X^{j,N}_{\tau_{n-}^{i,N}})}$.  
  We deduce that the process $\cal M$, defined by
  \begin{equation*}
    {\cal M}_t=\frac{1}{N}\sum_{i=1}^N{\sum_{n=1}^{A^{i,N}_t}{\left(\psi_{\tau^{i,N}_n}(X^{i,N}_{\tau^{i,N}_n})-
       \frac{1}{N-1}\sum_{j\neq i} {\psi_{\tau_{n-}^{i,N}}(
    X^{j,N}_{\tau_{n-}^{i,N}})}
        \right)}},\ \forall t\geq 0,
  \end{equation*}
  is a local martingale. We finally get
  \begin{equation}
    \label{EqTUsum2}
    \mu_t^{N}(\psi_t)=\mu^{N}_0(\psi_0)+\mathds{M}_t+{\cal
      M}_t+\frac{1}{N}\sum_{i=1}^N{\sum_{n=1}^{A^{i,N}_t}{
    \left[\frac{1}{N-1}\sum_{j\neq i} {\psi_{\tau_{n-}^{i,N}}(
    X^{j,N}_{\tau_{n-}^{i,N}})}\right]
    }},
  \end{equation}
  which is exactly \eqref{EqTUsum1} (we recall that $\mathds{M}$ and $\cal M$ implicitly depend on $N$).

\bi \textit{Step 2: proof of inequality~\eqref{eq:sec1-maj-before-normalization}.}\\
 The number of particles $N\geq 2$ is still fixed. Let us first prove that, for any $\alpha\in\{1,2,\ldots\}$,
$\nu^{N}_{T\wedge\tau^{N}_{\alpha}}(\psi_{T\wedge\tau^{N}_{\alpha}})-
\nu^{N}_0(\psi_0)$ is a local martingale. In order to simplify the notations and since $N\geq 2$ is fixed, we remove the superscripts $N$ until the end of this step when there is no risk of confusion.

\me 
The jump of the processes $(A_t)_{t\geq 0}$ and $\cal M$ only occurs at the rebirth times, thus we deduce from~\eqref{EqTUsum2} that
    \begin{multline*}
      \nu_T(\psi_T)-\nu_0(\psi_0)=\int_0^T\left(\frac{N-1}{N}\right)^{A_{t\minus}}d\mathds{M}_t
      - 
		\sum_{n=1}^{A_T} \left(\frac{N-1}{N}\right)^{A_{\tau_n\minus}}\left(\mathds{M}_{\tau_n}-\mathds{M}_{\tau_n\minus}\right)\\
      + 
      	\sum_{n=1}^{A_T}\nu_{\tau_n}(\psi_{\tau_n})-\nu_{\tau_n\minus}(\psi_{\tau_n\minus}).
    \end{multline*}
    Let us develop and compute each term $\nu_{\tau_n}(\psi_{\tau_n})-\nu_{\tau_n\minus}(\psi_{\tau_n\minus})$ in the right hand side. For any $n\geq 1$,
    \begin{multline*}
      \nu_{\tau_n}(\psi_{\tau_n})-\nu_{\tau_n\minus}(\psi_{\tau_n\minus})
      =\left(\frac{N-1}{N}\right)^{A_{\tau_n}}\left(\mu_{\tau_n}(\psi_{\tau_n})-\mu_{\tau_n\minus}(\psi_{\tau_n\minus})\right)\\
      +\mu_{\tau_n\minus}(\psi_{\tau_n\minus})\left(\left(\frac{N-1}{N}\right)^{A_{\tau_n}}-\left(\frac{N-1}{N}\right)^{A_{\tau_n\minus}}\right).
    \end{multline*}
On the one hand, denoting by $i$ the index of the killed particle at time $\tau_n$, we have
\begin{eqnarray*}
 \mu_{\tau_n}(\psi_{\tau_n})-\mu_{\tau_n\minus}(\psi_{\tau_n\minus})
   &=&\frac{1}{N(N-1)}\sum_{j\neq i} {\psi_{\tau_n^{i}\minus}(
    X^{j}_{\tau_n^{i}\minus})}
      + \mathds{M}_{\tau_n}-\mathds{M}_{\tau_n\minus}+{\cal M}_{\tau_n}-{\cal M}_{\tau_n\minus},
\end{eqnarray*}
where 
\begin{equation*}
\frac{1}{N(N-1)}\sum_{j\neq i} {\psi_{\tau_n^{i}\minus}(
    X^{j}_{\tau_n^{i}\minus})}=\frac{1}{N-1}\mu_{\tau_n\minus}(\psi_{\tau_n\minus})
      - \frac{1}{N(N-1)} \psi_{\tau_n\minus}(X^i_{\tau_n\minus})
\end{equation*}
and, by the definition of $\mathds{M}$,
\begin{equation*}
  - \frac{1}{N(N-1)} \psi_{\tau_n\minus}(X^i_{\tau_n\minus})=\frac{1}{N-1}\left(\mathds{M}_{\tau_n}-\mathds{M}_{\tau_n\minus}\right).
\end{equation*}
We then have
\begin{eqnarray*}
 \mu_{\tau_n}(\psi_{\tau_n})-\mu_{\tau_n\minus}(\psi_{\tau_n\minus})
   &=&\frac{1}{N-1}\mu_{\tau_n\minus}(\psi_{\tau_n\minus})
      + \frac{N}{N-1}\left(\mathds{M}_{\tau_n}-\mathds{M}_{\tau_n\minus}\right)+{\cal M}_{\tau_n}-{\cal M}_{\tau_n\minus},
\end{eqnarray*}
On the other hand, we have
\begin{eqnarray*}
  \left(\frac{N-1}{N}\right)^{A_{\tau_n}}-\left(\frac{N-1}{N}\right)^{A_{\tau_n\minus}}=-\frac{1}{N-1}\left(\frac{N-1}{N}\right)^{A_{\tau_n}}.
\end{eqnarray*}
Finally we get
\begin{eqnarray*}
  \nu_{\tau_n}(\psi_{\tau_n})-\nu_{\tau_n\minus}(\psi_{\tau_n\minus})
    &=&  \left(\frac{N-1}{N}\right)^{A_{\tau_n}\minus}\left(\mathds{M}_{\tau_n}-\mathds{M}_{\tau_n\minus}\right)
       + \left(\frac{N-1}{N}\right)^{A_{\tau_n}}\left({\cal M}_{\tau_n}-{\cal M}_{\tau_n\minus}\right).
\end{eqnarray*}
  We deduce that the process  $\nu_{t}(\psi_t)-
    \nu_0(\psi_0)$ is the sum of two local martingales and we have
  \begin{equation}
    \label{EqTUnudecomp1}
    \nu_{T}(\psi_{T})-
    \nu_0(\psi_0)=\int_0^{T}{\left(\frac{N-1}{N}\right)^{A_{t\minus}}d\mathds{M}_t}
    +\frac{N-1}{N}\int_{0}^{T}\left(\frac{N-1}{N}\right)^{A_{t\minus}}d{\cal
      M}_t
  \end{equation}

\me  Let us now bound both terms on the right-hand side of~\eqref{EqTUnudecomp1} (where $N$ is still
  fixed).  
In order to handle martingales instead of local martingales, we fix an integer $\alpha\geq 1$
  and we stop the particle system when the number of rebirths
  $A_t$ reaches $\alpha$, which is equivalent to stop the process at
  time $\tau_{\alpha}$ (which is actually $\tau^{N}_{{\alpha}}$). The processes $\mathds{M}$ and ${\cal M}$ stopped at
  time $\tau_{\alpha}$ are almost surely bounded by $\alpha \|f\|_{\infty}$.
  By the optional stopping
  time theorem, we deduce that $\mathds{M}$ and ${\cal M}$ stopped at
  time $\tau^{N}_{\alpha}$ are martingales.
On the one hand, the martingale jumps occuring at rebirth times ${\cal
    M}_{\tau_n}-{\cal M}_{\tau_n\minus}$ are bounded by
  $\|f\|_{\infty}/N$ (since its jumps are the difference between two positive terms bounded by $\|f\|_{\infty}/N$) and the martingale is constant between rebirth times, thus
  \begin{eqnarray}
    \nonumber
    E\left(\left| \frac{N-1}{N}\int_{0}^{T\wedge\tau_{\alpha}}\left(\frac{N-1}{N}\right)^{A_{t\minus}}d{\cal
          M}_t \right|^2\right)
      &\leq &
        E\Bigg[\sum_{n=1}^{A_T\wedge \alpha}
         {\left(\frac{N-1}{N}\right)^{2 A_{\tau_n\minus}}}{\left({\cal M}_{\tau_n}-{\cal M}_{\tau_n\minus}\right)^2}\Bigg]\\
      &\leq&\frac{\|f\|_{\infty}^2}{N}. \label{EqTUborneM1}
  \end{eqnarray}
  On the other hand, we have
  \begin{equation*}
    \begin{split}
      E\left(\left(\int_0^{T\wedge\tau_{\alpha}}{\left(\frac{N-1}{N}\right)^{A_{t-}}d\mathds{M}_t}\right)^2\right)
      &\leq
      E\left(\left(\mathds{M}_{T\wedge\tau_{\alpha}}\right)^2
      \right)\\
      &=\frac{1}{N^2}\sum_{i,j=1}^N
      E\left(\mathds{M}^{i}_{T\wedge\tau_{\alpha}}\mathds{M}^{j}_{T\wedge\tau_{\alpha}} \right)
    \end{split}
  \end{equation*}
  where
  \begin{equation*}
    E\left( \mathds{M}^{i}_{T\wedge\tau_{\alpha}}
      \mathds{M}^{j}_{T\wedge\tau_{\alpha}}
    \right)=\sum_{m=0,n=0}^{\alpha}E\left(\mathds{M}^{i,m}_{T\wedge\tau_{\alpha}}\mathds{M}^{j,n}_{T\wedge\tau_{\alpha}}\right).
  \end{equation*}
  If $i\neq j$, then the expectation of the product of the martingales
  $\mathds{M}^{i,n}$ and $\mathds{M}^{j,m}$ is $0$, since the
  particles are independent between the rebirths and do not undergo
  simultaneous rebirths. Thus we have
  \begin{equation*}
  	E\left( \mathds{M}^{i}_{T\wedge\tau_{\alpha}}
      \mathds{M}^{j}_{T\wedge\tau_{\alpha}}
    \right)=0,\ \forall i\neq j.
  \end{equation*}
  Assume now that $i=j$ and fix $m<n$. By definition, we have
  \begin{equation*}
    \mathds{M}^{i,m}_{T\wedge\tau_{\alpha}}=\mathds{M}^{i,m}_{T\wedge\tau_{\alpha}\wedge\tau^{i}_{m+1}},
  \end{equation*}
  which is measurable with respect to $\mathds{X}_{T\wedge
    \tau_{\alpha} \wedge \tau^{i}_{m+1}}$, then
  \begin{equation*}
    \begin{split}
      E\left( \mathds{M}^{i,m}_{T\wedge\tau_{\alpha}}
        \mathds{M}^{i,n}_{T\wedge\tau_{\alpha}} |
        \mathds{X}_{T\wedge\tau_{\alpha}\wedge \tau^{i}_{m+1}}
      \right)&= \mathds{M}^{i,m}_{T\wedge\tau_{\alpha}\wedge \tau^{i}_{m+1}}
      E\left(\mathds{M}^{i,n}_{T\wedge\tau_{\alpha}} | \mathds{X}_{T\wedge\tau_{\alpha}\wedge \tau^{i}_{m+1}}
      \right)\\
      &= \mathds{M}^{i,m}_{T\wedge\tau_{\alpha}\wedge \tau^{i}_{m+1}} \mathds{M}^{i,n}_{T\wedge\tau_{\alpha}\wedge \tau^{i}_{m+1}}=0,
    \end{split}
  \end{equation*}
  applying the optional sampling theorem with the martingale
  $\mathds{M}^{i,n}_{T\wedge\tau_{\alpha}}$ stopped at the uniformly bounded
  stopping time $T\wedge \tau_{\alpha} \wedge \tau^{i}_n$. We
  deduce that
  
  \me
  \begin{equation*}
    \begin{split}
      E\left(\left(\mathds{M}^{i}_{T\wedge\tau_{\alpha}}\right)^2\right)
      &=E\left(\sum_{n=0}^{\alpha}\left(\mathds{M}^{i,n}_{T\wedge\tau_{\alpha}}\right)^2\right)\\
      &\leq E\left( \sum_{n=0}^{A_{T\wedge \tau_{\alpha}}-1}\left( \psi_{\tau^{i,N}_n}\left(X^{i}_{\tau^i_n}\right) \right)^2 \right)
+ E\left( \left(\mathds{M}^{i,A_{T\wedge \tau_{\alpha}}}_{T\wedge\tau_{\alpha}}\right)^2 \right)\\
      &\leq E\left(\sum_{n=0}^{\alpha}
        \psi_{T\wedge\tau^{i}_n}(X^{i}_{T\wedge\tau^{i}_n})^2\right)+\|f\|_{\infty}^2\\
      &\leq \|f\|_{\infty}E\left(\sum_{n=0}^{\alpha}
        \psi^N_{T\wedge\tau^{i}_n}(X^{i}_{T\wedge\tau^{i}_n})\right)+\|f\|_{\infty}^2.
    \end{split}
  \end{equation*}
  
  \me
  By  \eqref{EqTUmartingale1}, we have
  \begin{equation*}
    E\left(\sum_{n=0}^{\alpha} \psi_{T\wedge\tau^{i}_n}(X^{i}_{T\wedge\tau^{i}_n})\right)\leq \|f\|_{\infty},
  \end{equation*}
thus
  \begin{equation*}
      E\left(\left(M^{i}_{T\wedge\tau_{\alpha}}\right)^2\right)\leq 2\|f\|_{\infty}^2.
  \end{equation*}
We finally have
  \begin{equation}
    \label{EqTUborneM2}
    E\left(\left(\int_0^{T\wedge\tau_{\alpha}}{\left(\frac{N-1}{N}\right)^{A_{t\minus}}d\mathds{M}_t}\right)^2\right) \leq \frac{2\|f\|_{\infty}^2}{N}.
  \end{equation}
  The decomposition~\eqref{EqTUnudecomp1} and inequalities \eqref{EqTUborneM1}
  and \eqref{EqTUborneM2} lead us to
  \begin{equation*}
    \sqrt{E\left(\left|\nu^{N}_{T\wedge\tau_{\alpha}}(P^N_{T-T\wedge
    \tau_{\alpha}}f)-
    \mu^{N}_0(P^N_{T\wedge\tau_{\alpha}}
    f)\right|^2\right)}\leq \frac{(1+\sqrt{2})\|f\|_{\infty}}{\sqrt{N}}.
  \end{equation*}
  By assumption, the number of rebirths of the $N$-particles system remains
  bounded up to time $T$ almost surely. We deduce that
  $T\wedge\tau^{N}_{\alpha}$ converges to $T$ when $\alpha$ goes to infinity (actually $T\wedge\tau^{N}_{\alpha}=T$ for $\alpha$ big enough
  almost surely). As a consequence, letting $\alpha$ go to infinity in the above inequality and using the dominated
  convergence theorem, we obtain
   \begin{equation}
         \label{EqTUconv1}
     \sqrt{E\left(\left|\nu^{N}_{T}(f)-
           \mu^{N}_0(P^N_{T} f)\right|^2\right)}\leq
     \frac{(1+\sqrt{2})\|f\|_{\infty}}{\sqrt{N}}.
   \end{equation}


\bi \textit{Step 3: conclusion of the proof of Theorem~\ref{theorem:convergence}.}\\
 By
 assumption, $\mu^{N}_0(P_T^N .)$ converges weakly to $\mu_0(P_T .)$. Thus, for any bounded continuous function $f:F\rightarrow\mathds{R}_+$, the sequence of random
 variables $\left(\mu^{N}_0(P^N_T \mathbf{1}_F), \mu^{N}_0(P^N_T
 f)\right)$ converges in distribution to the random variable
 $\left(\mu_0(P_T \mathbf{1}_F), \mu_0(P_T f)\right)$. By
 \eqref{EqTUconv1}, we deduce that the sequence of random variables
 $\left(\nu^{N}_T(\mathbf{1}_F),\nu^{N}_T(f)\right)$ converges in
 distribution to the random variable $\left(\mu_0(P_T \mathbf{1}_F),
 \mu_0(P_T f)\right)$. Now $\mu_0(P_T
 \mathbf{1}_F)$ never vanishes almost surely, thus
  \begin{equation*}
    \mu^{N}_T(f)=\frac{\nu^{N}_T(f)}{\nu^{N}_T(\1_F)}\xrightarrow[N\rightarrow\infty]{law}\frac{\mu_0(P_T
    f)}{\mu_0(P_T \1_F)},
  \end{equation*}
  for any bounded continuous function
  $f:F\rightarrow\mathds{R}_+$, which implies the first part of
  Theorem~\ref{theorem:convergence}.

 \me We also deduce from~\eqref{EqTUconv1} that
  \begin{equation*}
    \sqrt{E\left(\left|\left(\frac{N-1}{N}\right)^{A^{N}_T}-\mu^{N}_0\left(P^N_T
            \1_F\right)\right|^2\right)}\leq  \frac{1+\sqrt{2}}{\sqrt{N}},
  \end{equation*}
  then, using~\eqref{EqTUconv1} and the triangle inequality,
  \begin{equation*}
    \sqrt{E\left(\left|\mu^{N}_0(P^N_T \1_F) \mu^{N}_T(f)-
          \mu^{N}_0(P^N_T f)\right|^2\right)}\leq
    \frac{2(1+\sqrt{2})\|f\|_{\infty}}{\sqrt{N}}.
  \end{equation*}
  Using the Cauchy--Schwartz inequality, we finally deduce that
  \begin{equation*}
    E\left(\left|\mu^{N}_T(f)- \frac{\mu^{N}_0(P^N_T
          f)}{\mu^{N}_0\left(P^N_T \1_F\right)}\right|\right)\leq
    \sqrt{E\left(\frac{1}{\left(\mu^{N}_0\left(P^N_T\1_F\right)\right)^2}\right)}
    \frac{2(2+\sqrt{2})\|f\|_{\infty}}{\sqrt{N}},
  \end{equation*}
  which concludes the proof of Theorem~\ref{theorem:convergence}.
\end{proof}

\section{Criterion ensuring the non-explosion of the number of rebirths}
\label{section:non-explosion}

 In this section, we fix $N\geq 2$ and consider $N$-particles systems whose particles
evolve as 
independent 
diffusion processes and are subject to rebirth after hitting a boundary or after some exponential times. We begin by defining the diffusion processes which will drive the particles between the rebirths, then we define the jump measures giving the distribution of the rebirths
positions of the particles. As explained in the introduction, the interacting particle system will be well defined if and only if there is no accumulation of rebirths in finite time almost surely. The main result of this section is a criterion (given by Hypotheses~\ref{hypothesis:particle-dynamics} and~\ref{hypothesis:jump-measure} below) which ensures that the number of rebirths remains finite in finite time almost surely, so that the $N$-particles system is well defined at any time $t\geq 0$.


\bi For each $i\in\{1,\ldots,N\}$, let $E_i$ be an open subset of $\R^{d_i}$ ($d_i\geq 1$) and $D_i$ be an open subset of $\R^{d'_i}$ ($d'_i\geq 1$). Let 
$(e^i_{t},Z^i_{t})_{t\geq 0}$ be a \textit{time-inhomogeneous environment-dependent diffusion process} evolving in $E_i\times D_i$, 
where $e^{i}_t\in E_i$ denotes the state of the environment and
    $Z^{i}_t\in D_i$ denotes the actual position of the diffusion at time $t$.  Each diffusion
 process $(e^i_{t},Z^i_{t})_{t\geq 0}$ will be used to define the dynamic of the
 $i^{th}$ particle of the system between its rebirths.
By a \textit{time-inhomogeneous environment-dependent diffusion
  process}, we mean that, for any
$i\in\{1,...,N\}$, there exist four measurable functions
 \begin{equation*}
   \begin{split}
     &s_i:[0,+\infty[\times E_i\times D_i\mapsto
     \mathbb{R}^{d_i}\times\mathbb{R}^{d_i}\\
     &m_i:[0,+\infty[\times E_i\times D_i\mapsto
     \mathbb{R}^{d_i}\\
     &\sigma_i:[0,+\infty[\times E_i\times D_i\mapsto
     \mathbb{R}^{d'_i}\times\mathbb{R}^{d'_i}\\
     &\eta_i:[0,+\infty[\times E_i\times D_i\mapsto \mathbb{R}^{d'_i},
  \end{split}
 \end{equation*}
 such that $(e^{i},Z^{i})$ is solution to the
 stochastic differential system
 \begin{equation*}
   \begin{split}
     de^{i}_t&=s_i(t,e^{i}_t,Z^{i}_t)d\beta^{i}_t+m_i(t,e^{i}_t,Z^{i}_t)dt,\ \quad e^{i}_0\in E_i,\\
     dZ^{i}_t&=\sigma_i(t,e^{i}_{t},Z^i_{t})dB^{i}_t+\eta^{i}({t},e^{i}_{t},Z^{i}_{t})
     dt,\ \quad Z^{i}_0\in D_i,
   \end{split}
 \end{equation*}
 where the $(\beta^{i},B^{i})$, $i=1,\ldots,N$, are independent standard $d_i+d'_i$ Brownian
 motions.  We assume that the process $(e^i_{t},Z^i_{t})_{t\geq 0}$ is subject to two different kind of killings: it is killed  when $Z^{i}_t$ hits $\partial D_i$ (\textit{hard killing}) and with a rate of killing
 $\kappa_i(t,e^{i}_t,Z^{i}_t)\geq 0$ (\textit{soft killing}), where
 \begin{equation*}
   \kappa_i:[0,+\infty[\times E_i\times D_i\mapsto \mathbb{R}_+
 \end{equation*}
 is a uniformly bounded measurable function. 
 When a process $(e^i,Z^i)$ is killed, it is immediately sent to a cemetery point $\partial\notin E_i \times D_i$ where it remains forever. In particular,  the killed process is c\`adl\`ag with values in $E_i \times D_i\cup \{\partial\}$.

\bi We assume that we're given two measurable functions
 \begin{equation*}
   {\cal
   S}:[0,+\infty[\times E_1\times\cdots\times E_N \times D_1\times\cdots\times D_N\rightarrow {\cal
       M}_1(E_1\times\cdots\times E_N \times D_1\times\cdots\times D_N)
 \end{equation*}
 and
 \begin{equation*}
   {\cal H}:[0,+\infty[\times
         E_1\times\cdots\times E_N \times \partial (D_1\times\cdots\times D_N)\rightarrow {\cal
           M}_1(E_1\times\cdots\times E_N\times D_1\times\cdots\times D_N),
 \end{equation*}
 where ${\cal M}_1(E_1\times\cdots\times E_N\times D_1\times\cdots\times D_N)$ denotes the space of probability measures on $E_1\times\cdots\times E_N\times D_1\times\cdots\times D_N$. These measures will be used to define the distribution of the rebirths locations of the particles, respectively for soft and hard killings.

 \bi 
  The $N$-particles system is denoted by $(\mathbb{O}_t,\mathbb{X}_t)_{t\geq 0}$, where $\mathbb{O}_t=(o^{1}_t,\cdots,o^{N}_t)\in E_1\times\cdots\times E_N$ is the vector of environments at time $t$ and $\mathbb{X}_t=(X^{1}_t,\cdots,X^{N}_t)\in D_1\times\cdots\times D_N$ is the vector of positions of the particles  at time $t$. In particular, the position of the $i^{th}$ particle of the system at time $t\geq 0$ is given by $(o^{i}_t,X^{i}_t)\in E_i\times D_i$.
 
\me \textbf{Dynamic of the particle system} 
 
 \noindent The particles of the
 system evolve
 as independent copies of $(e^i,Z^i)$, $i=1,...,N$, until one of
 them is killed. Then
 \begin{itemize}
 \item[--] if two or more particles are simultaneously killed at this time, we stop the definition of the particle system and say that it undergoes a \textit{failure}. This time is denoted by $\tau_{stop}$.
 \item[--] otherwise,
 \begin{itemize} 
 \item[-] if the unique killed particle is softly killed, then the whole system
 jumps instantaneously with respect to the jump measure ${\cal
   S}(t,\mathbb{O}_{t\minus},\mathbb{X}_{t\minus})$
   \item[-] if the unique killed particle is hardly
 killed, then the whole system
 jumps instantaneously with respect to ${\cal
   H}(t,\mathbb{O}_{t\minus},\mathbb{X}_{t\minus})$
   \end{itemize}
   and, in both cases, we say that the system undergoes a \textit{rebirth}.
 \end{itemize}
 After this operation and in the case of a rebirth, all the particles lie in $F$. Then they evolve
 as independent copies of $(e^i,Z^i)$, $i=1,...,N$, until one of
 them is killed (which leads to a failure or a rebirth) and so on.

 \bi If the process does never fail, we set $\tau_{stop}=+\infty$. We denote the successive rebirth times of the particle system by
 \begin{equation*}
   \tau_1 < \tau_2 < ... <\tau_n <...
 \end{equation*}
 and we set $\tau_{\infty}=\lim_{n\rightarrow\infty} \tau_n$. Since there is no natural way to define the particle system after time $\tau_{\infty}\wedge \tau_{stop}$, the $N$-particles system is well defined at any time $t\geq 0$ if and only if  $\tau_{\infty}\wedge \tau_{stop}=+\infty$ almost surely. In Theorem~\ref{ThExistence} below, we state that this happens if Hypotheses~\ref{hypothesis:particle-dynamics} and~\ref{hypothesis:jump-measure} below are fulfilled; this is the main result of this section. We emphasize that the above construction and the assumptions of Theorem~\ref{ThExistence} include the particular case of the Fleming-Viot type system studied in Section~\ref{section:approximation}.

\bi The first assumption concerns the processes $(e^i,Z^i)$, $i=1,\ldots,N$, which drive the particles between the rebirths. We denote by $\phi_{i}$ the Euclidean distance to
 the boundary $\partial D_i$, defined for all $x\in \mathbb{R}^{d'_i}$ by
 \begin{equation*}
   \phi_i(x)=\inf_{y\in \partial D_i} \|x-y\|_2,
 \end{equation*}
 where $\|.\|_2$ denotes the Euclidean norm on $\mathbb{R}^{d'_i}$. For any $a>0$ and any $i\in\{1,\ldots,N\}$, we define the boundary's neighbourhood $D_i^a$ by
 $$
 	D_i^a=\{x\in D_i,\ \phi_i(x)<a\}.
 $$

\begin{hypothesis}
  \label{hypothesis:particle-dynamics}
  We assume that there exist five positive constants $a_0$, $k_g$, $c_{0}$ and $C_{0}$ such that
  \begin{enumerate}
  \item $\phi_i$ is of class $C^2$ on $D_i^{a_0}$, with uniformly bounded derivatives,
  \item $\kappa_{i}$ is uniformly bounded on $[0,+\infty[\times E_i
      \times D_i$ and $s_i,\sigma_i,m_i$ and $\eta_i$
      are uniformly bounded on $[0,+\infty[\times E_i\times D_i^{a_0}$,
      \item for any  $i\in\{1,...,N\}$, there exist two
    measurable functions $f_i:[0,+\infty[\times E_i\times D^{a_0}_i\rightarrow
    \mathbb{R}_+$ and $g_i:[0,+\infty[\times E_i\times D^{a_0}_i\rightarrow
    \mathbb{R}$ such that $\forall (t,e,z)\in [0,+\infty[\times
    E_i\times D^{a_0}_i$,
    \begin{equation}
      \label{chapitre4:EqHyThExDriftTerm}
      \sum_{k,l}\frac{\partial \phi_i}{\partial x_k}(z) \frac{\partial
        \phi_i}{\partial x_l}(z)
      [\sigma_i\sigma^{*}_i]_{kl}(t,e,z)=f_i(t,e,z)+g_i(t,e,z),
    \end{equation}
    and such that
    \begin{enumerate}
    \item $f_i$ is of class $C^1$ in time and of class $C^2$ in
      environment/space on $[0,+\infty[\times E_i\times D^{a_0}_i$, with uniformly bounded derivatives,
    \item for all
      $(t,e,z)\in[0,+\infty[\times E_i \times D^{a_0}_i$,
      \begin{equation*}
	|g_i(t,e,z)|\leq k_g\phi_i(z),
      \end{equation*}
    \item for all $(t,e,z)\in[0,+\infty[\times E_i \times D^{a_0}_i$,
      \begin{equation*}
        c_{0}<f_i(t,e,z)<C_{0}\text{ and }c_{0}<f_i(t,e,z)+g_i(t,e,z)<C_{0}.
      \end{equation*}
    \end{enumerate}
  \end{enumerate}
\end{hypothesis}

\me \textit{Remark 5.} We recall that the $C^k$ regularity of $\phi_i$ near the
    boundary is equivalent to the $C^k$ regularity of the boundary
    $\partial D_i$ itself, for all $k\geq 2$ (see \cite[Chapter 5, Section
    4]{Delfour2001}).

\smallskip\noindent \textit{Remark 6.} The euclidean distances $\phi_i$ in Hypothesis~\ref{hypothesis:particle-dynamics} could be replaced by other functions with similar properties, in order to weaken the assumption. However, for sake of clarity, we do not develop this possibility and refer the reader to~\cite{Grigorescu2011} for an alternative class of \textit{distance like} functions.

\bi Our second assumption is related to the rebirth measures ${\cal H}$ and ${\cal S}$. We emphasize that if a rebirth due to hard killing happen, then at most one particle hits its associated boundary $\partial D_i$. This implies that, at a rebirth time, the whole set of particles hits one and only one of the sets ${\cal D}_i$, $i=1,\cdots,N$, defined by
\begin{equation*}
{\cal D}_i=\left\lbrace (x_1,\cdots,x_N)\in\partial (D_1\times\ldots\times D_N),\ x_i\in \partial D_i\ \mbox{and}\ x_j\in D_j,\forall j\neq i \right\rbrace.
\end{equation*}
With this definition, it is clear that the $i^{th}$ particle undergoes a hard killing leading to a rebirth if and only if $(\cdot,\mathbb{O},\mathbb{X})$ hits $[0,+\infty[\times (E_1\times\ldots\times E_N)\times {\cal D}_i$. In particular, the values taken by ${\cal H}$ outside these sets does not influence the behaviour of the particle system.

 \begin{hypothesis}
 \label{hypothesis:jump-measure}
 We assume that
     \begin{enumerate}
  \item There exists a non-decreasing continuous function
    $h:\mathbb{R}_+\rightarrow\mathbb{R}_+$ vanishing only at $0$ such that,
    $\forall i\in\{1,...,N\}$,
    \begin{equation*}
      \inf_{(t,e,(x_1,...,x_N))\in[0,+\infty[\times (E_1\times\ldots\times E_N)\times {\cal D}_i}{\cal H}(t,e,x_1,...,x_N)\left((E_1\times\ldots\times E_N)\times A_{i}\right)\geq p_0,
    \end{equation*}
    where $p_0>0$ is a positive constant and $A_{i}\subset D_1\times\ldots\times D_N$ is the
    set defined by
    \begin{equation*}
      A_{i}=\left\lbrace (y_1,\ldots,y_N)\mid\exists j\neq i\text{ such that }\phi_{i}(y_{i})\geq h(\phi_j(y_j)) \right\rbrace.
    \end{equation*}
    
  \item We have
    \begin{equation*}
      \inf_{(t,e,x)\in[0,+\infty[\times (E_1\times\cdots\times E_N)\times {\cal D}_i}{\cal H}(t,e,x_1,...,x_N)(B_{e,x})=1,
    \end{equation*}
    where $B_{e,x}\subset (E_1\times\ldots\times E_N)\times(D_1\times\ldots\times D_N)$ is defined by
    \begin{equation*}
      B_{e,x}=\left\lbrace (e',x')\,|\,\forall j,\  \phi_{j}(x'_{j})\geq \phi_j(x_j)\left(1\vee\frac{\sqrt{f_j(t,e'_j,x'_j)}}{\sqrt{f_j(t,e_j,x_j)}}\right)  \right\rbrace,
    \end{equation*}
    $f_j$ being the function of Hypothesis~\ref{hypothesis:particle-dynamics}, extended outside $[0,+\infty[\times E_j\times D^{a_0}_j$ by the value $c_0$.
  \end{enumerate}
 \end{hypothesis}
 
\medskip \noindent Let us explain Hypothesis~\ref{hypothesis:jump-measure}.
 
\medskip  \noindent - The set $A_i$ is a subset of $D_1\times\ldots\times D_N$ such that if the $i^{th}$ component of an element $(y_1,\cdots,y_n)\in A_i$ is near the boundary $\partial D_i$ (that is if $\phi_i(y_i)\ll 1$), then at least one another component $y_j$ fulfils $h(\phi_j(y_j))\ll 1$ and thus $\phi_D(y_j)\ll 1$. This implies that if the rebirth position of the $i^{th}$ particle after a hard killing is located near the boundary $\partial D_i$, then, with a probability lowered by $p_0$, at least one other particle is located near the boundary.

\medskip \noindent - The set $B_{e,x}$ is a subset of $E_1\times\ldots \times E_N\times D_1\times\ldots \times D_N$ such that for any $(e',x')\in B_{e,x}$, the component $x'_j$ lies farther from the boundary than $x_j$, for any $j\in\{1,\ldots,N\}$. This means that after a hard killing, the rebirth position of any particle lies farther from the boundary after the rebirth than before the rebirth. 

\bi \textit{Remark 7.} Hypothesis~\ref{hypothesis:jump-measure} includes the case studied
      by Grigorescu and Kang in \cite{Grigorescu2011} (without environment dependency or time inhomogeneity), where
      \begin{equation*}
        \label{s11_e8}
        {\cal H}= \sum_{j\neq i} p_{ij}(x_i) \delta_{x_j},\ \forall (x_1,...,x_N)\in {\cal D}_i.
      \end{equation*}
      with $\sum_{j\neq i} p_{ij}(x_i)=1$ and $\inf_{i\in\{1,...,N\},j\neq
        i, x_i\in \partial D} p_{ij}(x_i)>0$.  Indeed, in that case, the rebirth position of any killed particle is the position of another particle. It
      implies that Hypothesis~\ref{hypothesis:jump-measure} is fulfilled with $p_0=1$ and
      $h(\phi_j)=\phi_j$. This case also includes the Fleming-Viot type
      system of the approximation method proved in Section~\ref{section:approximation}.

\bi We're now able to state the main result of this section.
\begin{theorem}
  \label{ThExistence}
  If Hypotheses~\ref{hypothesis:particle-dynamics} and \ref{hypothesis:jump-measure}
  are satisfied, then $\tau_{stop}\wedge\tau_{\infty}=+\infty$ almost surely. In particular, the $N$-particles system with rebirth $(\cdot,\mathbb{O},\mathbb{X})$ is well defined at any time $t\geq 0$.
\end{theorem}

\bi \textit{Remark 8.} 
An other interesting example of diffusion processes with killing is given by reflected ones (see for instance~\cite{Kolb2010}
 and~\cite{Zhen2010}). Thus
 a natural question is whether our non-explosion result holds true when the diffusion process $(e^i,Z^i)$ is reflected on $\partial D_i$ and killed when its local time on $\partial D_i$ reaches the value
 of an independent exponential random variable. In fact, the only difference lies in the It\^o calculus developed in the end of this section, where a local time term should appear due to reflection.  The generality of Proposition~\ref{PrHitting} (which concludes the proof) allows to handle this additional term. As a consequence, our proof would remain perfectly valid under the setting of reflected diffusion processes.
 
 \begin{proof}[Proof of Theorem~\ref{ThExistence}]   Since the rate of killing $\kappa_i$ is assumed to be uniformly bounded for all $i\in\{1,...,N\}$, there is no accumulation of soft killing
  events almost surely. As a consequence, we only have to show the
  non-accumulation of hard killing events and can assume up to the end
  of the proof that $\kappa_i=0$ for all $i\in\{1,...,N\}$.

\bi
In Subsection~\ref{subsection:attainability}, we prove that, conditionally to $\tau_{stop}\wedge \tau_{\infty}<+\infty$, there exists $i\neq j\in\{1,\ldots,N\}$ such that the pair of semi-martingales $(\phi_i(X^i_{t}),\phi_j(X^j_{t}))$ converges
to $(0,0)$ when $t$ goes to $\tau_{stop} \wedge \tau_{\infty}$. 
This leads us to
  \begin{equation*}
    P\left(\tau_{stop}\wedge\tau_{\infty}<+\infty\right)\leq 
      \sum_{1\leq i < j \leq N} P\left(T^{ij}_0<+\infty\right),
  \end{equation*}
where $T_0^{ij}$ is the stopping time defined by
\begin{equation*}
  T_0^{ij}= \inf\left\lbrace t\in[0,+\infty[,\ \exists s_n\xrightarrow[n\rightarrow\infty]{}t\text{ such that }\lim_{n\rightarrow\infty}(\phi_i(X^i_{s_n}),\phi_j(X^j_{s_n}))=(0,0) \right\rbrace,
\end{equation*}
where $(s_n)$ runs over the set of non-decreasing sequences. 

\bi In a second step, we conclude the proof of Theorem~\ref{ThExistence} by proving that, for any $i\neq j$,
\begin{equation*}
P(T^{ij}_0<+\infty)=0.
\end{equation*}
The proof of this assertion is itself divided into two parts: in Subsection~\ref{subsection:decomposition}, we compute in details the It\^o's decomposition of $\phi_i(X^i)$ and $\phi_j(X^j)$; in Section~\ref{section:non-attainability}, we prove that the just obtained It\^o's decomposition implies the non-attainability of $(0,0)$ for $(\phi_i(X^i),\phi_j(X^j))$, that is $T_0^{ij}=+\infty$ almost surely, which concludes the proof.

\bi\textit{Remark 9.}
The non-attainability result used in the proof has its own interest and is thus stated as Proposition~\ref{PrHitting} in the independent Section~\ref{section:non-attainability}.

\subsection{Explosion of the number of rebirths implies simultaneous convergence to the boundary}
\label{subsection:attainability}
Let us prove that, under Hypotheses~\ref{hypothesis:particle-dynamics} and~\ref{hypothesis:jump-measure}
\begin{equation}
\label{equation:attainability}
    P\left(\tau_{stop}\wedge\tau_{\infty}<+\infty\right)\leq 
      \sum_{1\leq i < j \leq N} P\left(T^{ij}_0<+\infty\right).
\end{equation}

\bi  If $\tau_{stop}<+\infty$, then at least two particles
  $X^{i}$ and $X^{j}$ hit their respective boundaries at time
  $\tau_{stop}$ (by definition of this stopping time). It implies that $\phi_i(X^i_{\tau_{stop}\minus})=\phi_j(X^j_{\tau_{stop}\minus})=0$.
   We deduce that
  \begin{equation}
  \label{equation:tau-stop-implies-explosion}
  \left\lbrace \tau_{stop}<+\infty\right\rbrace\subset 
      \bigcup_{1\leq i < j \leq N} \left\lbrace T^{ij}_0<+\infty\right\rbrace.
  \end{equation}
Define
the event 
$
    {\cal E}=\{\tau_{\infty}<+\infty\text{ and }  \tau_{stop}=+\infty\}
$. It remains us to prove that, up to a negligible event,
\begin{equation}
\label{equation:e-included-in-attainability}
{\cal E}\subset \bigcup_{1\leq i < j \leq N} \left\lbrace T^{ij}_0<+\infty\right\rbrace.
\end{equation}

\me  Conditionally to ${\cal E}$, the total number of rebirths
  converges to $+\infty$ up to time $\tau_{\infty}$. Since there is only a
  finite number of particles, at least one of them, say $i_0$, is killed
  infinitely many times up to time $\tau_{\infty}$. We denote by $(\tau^{i_0}_n)_n$ the sequence of rebirth times of the particle $i_0$ (we thus have $\tau^{i_0}_n<\tau_{\infty}<\infty$ for all $n$). In a first step, we prove that $\phi_{i_0}(X^{i_0}_{\tau^{i_0}_n})$ converges to $0$ when $n\rightarrow\infty$. In a second step, we prove that there exists a random index $j_0\neq i_0$ such that $\phi_{j_0}(X^{j_0}_{\tau^{i_0}_n})$ converges to $0$ when $n\rightarrow\infty$. This will imply~\eqref{equation:e-included-in-attainability} and thus Inequality~\eqref{equation:attainability}.

\me \textit{Step 1: convergence of $\phi_{i_0}(X^{i_0}_{\tau^{i_0}_n})$ to $0$ when  $n\rightarrow\infty$.}\\
 We prove that, up to a negligible event,
\begin{equation}
\label{equation:tau-infty-implies-convergence-i-0}
{\cal E}\subset\left\lbrace\phi_{i_0}(X^{i_0}_{\tau_n^{i_0}})\xrightarrow[n\rightarrow{\infty}]{} 0\right\rbrace.
\end{equation} 
 It\^o's formula will be a very useful tool in our proof. However, since $\phi_{i_0}$ is of class $C^2$ only on the set $D^{a_0}_{i_0}$, one can not apply directly the It\^o's formula to $\phi_{i_0}(X^{i_0}_{\cdot})$ . In order to overcome this difficulty, let $\gamma:\R_+\mapsto\R_+$ be a $C^2$ function such that
  \begin{equation*}
  	\left\lbrace
  	\begin{array}{l}
  	\gamma\text{ is non-decreasing with uniformly bounded derivatives,}\\
  	\gamma(u)=u,\ \forall u\in[0,\frac{a_0}{4}[,\\
  	\gamma(u)=\frac{a_0}{2},\ \forall u\in[a_0,+\infty[.
  	\end{array}
  	\right.
  \end{equation*}
The function $\gamma\circ\phi_{i_0}$ is of class $C^2$ on the whole set $D_{i_0}$ with derivatives
\begin{equation*}
\frac{\partial \gamma\circ\phi_{i_0}}{\partial x_l}(x)=
\left\lbrace
\begin{array}{l}
\gamma'(\phi_{i_0}(x))\frac{\partial \phi_{i_0}}{\partial x_l}(x),\text{ if }x\in D^{a_0}_{i_0}\\
0\text{ otherwise,}
\end{array}
\right.
\end{equation*}
and
\begin{equation*}
\frac{\partial^2 \gamma\circ\phi_{i_0}}{\partial x_l\partial x_k}(x)=
\left\lbrace
\begin{array}{l}
\gamma''(\phi_{i_0}(x))\frac{\partial \phi_{i_0}}{\partial x_k}(x)+\gamma'(\phi_{i_0}(x))\frac{\partial^2 \phi_{i_0}}{\partial x_l\partial x_k}(x),\text{ if }x\in D^{a_0}_{i_0},\\
0,\text{ otherwise,}
\end{array}
\right.
\end{equation*}
where $k,l\in\{1,\ldots,d'_{i_0}\}$. By definition of $\gamma$ and by Hypothesis~\ref{hypothesis:particle-dynamics}, these derivatives are uniformly bounded over $D_{i_0}$ and one can apply the It\^o's formula to $\gamma\circ\phi_{i_0}(X^{i_0}_{\cdot})$.

\me  Applying It\^o's formula to $\gamma\circ\phi_{i_0}(X^{i_0}_{\cdot})$ between times $\tau^{i_0}_n$ and $\tau^{i_0}_{n+1}$, we obtain
\begin{equation*}
\begin{split}
\gamma\left(\phi_{i_0}(X^{i_0}_{\tau^{i_0}_{n+1}\minus})\right)-\gamma\left(\phi_{i_0}(X^{i_0}_{\tau^{i_0}_n})\right)&=\sum_{l=1}^{d'_{i_0}}\int_{\tau^{i_0}_n}^{\tau^{i_0}_{n+1}} \frac{\partial \gamma\circ \phi_{i_0}}{\partial x_l}(X^{i_0}_{t\minus})\left[\sigma_{i_0}(t\minus,o^{i_0}_{t\minus},X^{i_0}_{t\minus}) dB^{i_0}_t\right]_l\\
&+\frac{1}{2}\sum_{l,k=1}^{d'_{i_0}} \int_{\tau^{i_0}_n}^{\tau^{i_0}_{n+1}} \frac{\partial^2 \gamma\circ\phi_{i_0}}{\partial x_l\partial x_k}(X^{i_0}_{t\minus})
\left[\sigma_{i_0}\sigma_{i_0}^*\right]_{kl}(t\minus,o^{i_0}_{t\minus},X^{i_0}_{t\minus}) dt\\
&+\sum_{l=1}^{d'_{i_0}} \int_{\tau^{i_0}_n}^{\tau^{i_0}_{n+1}}\frac{\partial \gamma\circ\phi_{i_0}}{\partial x_l}(X^{i_0}_{t\minus})\left[\mu_{i_0}(t\minus,o^{i_0}_{t\minus},X^{i_0}_{t\minus})\right]_l dt\\
&+\sum_{\tau^{i_0}_n < t < \tau^{i_0}_{n+1}} \gamma\left(\phi_{i_0}(X^{i_0}_{t})\right)-\gamma\left(\phi_{i_0}(X^{i_0}_{t\minus})\right).
\end{split}
\end{equation*}
By Hypothesis~\ref{hypothesis:jump-measure}, $\phi_{i_0}(X^{i_0})$ can only have non-negative jumps. Since $\gamma$ is non-decreasing, we deduce that the last sum is non-negative. Since $\tau^{i_0}_{n+1}-\tau^{i_0}_n$ converges to $0$ and the derivatives of $\gamma\circ \phi_{i_0}$ are uniformly bounded, we also deduce that the integrals in the right term converge to $0$ when $n$ goes to $+\infty$. We thus have, conditionally to the event $\cal E$,
\begin{equation}
\label{equation:liminf}
\liminf_{n\rightarrow\infty} \gamma\left(\phi_{i_0}(X^{i_0}_{\tau^{i_0}_{n+1-}})\right)-\gamma\left(\phi_{i_0}(X^{i_0}_{\tau^{i_0}_n})\right) \geq 0.
\end{equation}
But $\tau^{i_0}_{n+1}$ is by definition a rebirth time for $X^{i_0}$ and thus $\phi_{i_0}(X^{i_0}_{\tau^{i_0}_{n+1}\minus})=0$ (we recall that we assumed without loss of generality that $\kappa_i=0$, so that all rebirths are due to hard killing). Since the continuous function $\gamma$ is equal to $0$ at $0$, we deduce that
\begin{equation*}
\gamma\left(\phi_{i_0}(X^{i_0}_{\tau^{i_0}_{n+1}-})\right)-\gamma\left(\phi_{i_0}(X^{i_0}_{\tau^{i_0}_n})\right)=-\gamma\left(\phi_{i_0}(X^{i_0}_{\tau^{i_0}_n})\right)\leq 0.
\end{equation*}
This and the limit~\eqref{equation:liminf} imply that
\begin{equation*}
\lim_{n\rightarrow\infty} \gamma\left(\phi_{i_0}(X^{i_0}_{\tau^{i_0}_n})\right)=0,
\end{equation*}
which implies~\eqref{equation:tau-infty-implies-convergence-i-0}.

\bi \textit{Step 2: there exists $j_0\neq i_0$ such that $\phi_{j_0}(X^{j_0}_{\tau^{i_0}_n})$ converges to $0$.}\\
We prove now that, up to a negligible event,
  \begin{equation}
  \label{equation:tau-infty-implies-j-0}
{\cal E}\subset\left\lbrace\exists j_0\neq i_0,\  \phi_{i_0}(X^{i_0}_{\tau^{i_0}_n})\geq h(\phi_{j_0}(X^{j_0}_{\tau^{i_0}_n})),\ \text{for
   infinitely many}\ n\geq 1\right\rbrace,
  \end{equation}  
where $h$ is the continuous function of Hypothesis~\ref{hypothesis:jump-measure}. We denote by ${\cal A}_n$ the event
  \begin{equation*}
    \label{s12_e7}
          {\cal A}_{n}=\left\lbrace \tau^{i_0}_n<+\infty\text{ and, }\forall j\neq i_0,\text{ }
          \phi_{i_0}(X^{i_0}_{\tau_n^{i_0}})<
          h(\phi_{j}(X^{j}_{\tau_n^{i_0}}))
          \right\rbrace.
  \end{equation*}
  We have,
  for all $1\leq k\leq l$,
  \begin{align*}
    P\left( \bigcap_{n=k}^{l+1}{{\cal A}_n} \right) &=E\left(E\left(
    \prod_{n=k}^{l+1}\mathds{1}_{{\cal A}_n}\,|\,(X^1_t,...,X^N_t)_{0\leq
    t<\tau_{l+1}^{i_0}} \right) \right)\\
    &=E\left(\prod_{n=k}^{l}\mathds{1}_{{\cal A}_n}
    E\left(\mathds{1}_{{\cal A}_{l+1}}\,|\,(X^1_t,...,X^N_t)_{0\leq
    t<\tau_{l+1}^{i_0}} \right) \right).
  \end{align*}
  By construction of the rebirth mechanism of the interacting particle
  system and by the first point of Hypothesis \ref{hypothesis:jump-measure},
  \begin{align*}
    E\left(\mathds{1}_{{\cal A}_{l+1}}\,|\,(X^1_t,...,X^N_t)_{0\leq
      t<\tau_{l+1}^{i_0}} \right) &=\1_{\tau^{i_0}_{l+1}<+\infty} {\cal
      H}(\tau^{i_0}_{l+1},\mathds{O}_{\tau_{l+1-}^{i_0}},\mathds{X}_{\tau_{l+1-}^{i_0}})\left(
    {A_{i_0}^c} \right)\\ &\leq 1-p_0,
  \end{align*}
  where $A_{i_0}$ and $p_0$ are defined in Hypothesis \ref{hypothesis:jump-measure}.
  By induction on $l$, we get
  \begin{equation*}
    P\left( \bigcap_{n=k}^{l}{{\cal A}_n} \right)\leq (1-p_0)^{l-k},\ \forall 1\leq k \leq l.
  \end{equation*}
  Since $p_0>0$, it yields that
  \begin{equation*}
    P\left( \bigcup_{k\geq 1}\bigcap_{n=k}^{\infty}{{\cal A}_n} \right)=0.
  \end{equation*}
  As a consequence, if $\cal E$ happens, then, for infinitely many rebirth times $\tau^{i_0}_n$, one
  can find a particle $j\neq i_0$ such that $\phi_{i_0}(X^{i_0}_{\tau^{i_0}_n})\geq
  h(\phi_{j}(X^{j}_{\tau^{i_0}_n}))$. Since there is only a finite number of
  other particles, one can find a particle, say $j_0$ (which is a random
  index), such that
  \begin{equation*}
{\cal E}\subset \left\lbrace\phi_{i_0}(X^{i_0}_{\tau^{i_0}_n})\geq h(\phi_{j_0}(X^{j_0}_{\tau^{i_0}_n})),\ \text{for
   infinitely many}\ n\geq 1\right\rbrace.
  \end{equation*}
Let us now conclude the proof of~\eqref{equation:e-included-in-attainability}.
We deduce from~\eqref{equation:tau-infty-implies-convergence-i-0} and Step~2 that
\begin{align*}
{\cal E}&\subset \left\lbrace h(\phi_{j_0}(X^{j_0}_{\tau^{i_0}_n}))\xrightarrow[n\rightarrow\infty]{} 0,\ \text{up to a subsequence}\right\rbrace,\\
&\subset \left\lbrace \phi_{j_0}(X^{j_0}_{\tau^{i_0}_n})\xrightarrow[n\rightarrow\infty]{} 0,\ \text{up to a subsequence}\right\rbrace.
\end{align*}
 Using again~\eqref{equation:tau-infty-implies-convergence-i-0}, we obtain that, up to a negligible event,
  \begin{align*}
  {\cal E}&\subset \left\lbrace\exists j_0\neq i_0,\  \lim_{n\rightarrow\infty}{(\phi_{i_0}(X^{i_0}_{\tau^{i_0}_n}),\phi_{j_0}(X^{j_0}_{\tau^{i_0}_n}))}=(0,0)\ \text{up to a subsequence}\right\rbrace\\
  &\subset \bigcup_{\substack{i,j=1\\ i\neq j}}^N\left\{ T_0^{ij}<+\infty \right\}.
  \end{align*}
This and~\eqref{equation:tau-stop-implies-explosion} clearly implies~\eqref{equation:attainability}.

\subsection{Non-convergence to $(0,0)$ for $(\phi_i(X^i),\phi_j(X^i))$}
\label{subsection:decomposition}
  \bi   Fix $i\neq j\in\{1,...,N\}$. In this subsection, we prove that
  \begin{equation}
  \label{equation:non-attainability-to-prove1}
  P\left(T_0^{ij}<+\infty\right)=0,
  \end{equation}
   which will conclude the proof of Theorem~\ref{ThExistence}.
  
  \me Let us introduce 
  the sequence of stopping times $(t_{n})_{n\in\mathbb{N}}$ , defined by
    \begin{equation*}
    t_0=\inf\{t\in [0,\tau_{stop}\wedge\tau_{\infty}[,\ \sqrt{\phi_i(X^i_t)^2+\phi_j(X^j_t)^2}\leq a_0/2\}
  \end{equation*}
  and, for all $n\geq 0$,
  \begin{equation*}
    \begin{split}
      t_{2n+1}&=\inf\{t\in [t_{2n},\tau_{stop}\wedge\tau_{\infty}[,\ \sqrt{\phi_i(X^i_t)^2+\phi_j(X^j_t)^2}\geq a_0\}\\
      t_{2n+2}&=\inf\{t\in [t_{2n+1},\tau_{stop}\wedge\tau_{\infty}[,\ \sqrt{\phi_i(X^i_t)^2+\phi_j(X^j_t)^2}\leq a_0/2\}.
    \end{split}
  \end{equation*}
By definition of the sequence $(t_n)_{n\geq 1}$, we clearly have
    \begin{equation*}
    \left\lbrace
    \begin{array}{l}
      \phi_i(X^i_t)< a_0\text{ and }\phi_j(X^j_t)<a_0,\ \forall t\in\bigcup_{n\geq 0}[t_{2n},t_{2n+1}[,\\
      \sqrt{\phi_i(X^i_t)^2+\phi_j(X^j_t)^2}\geq a_0/2 \text{ otherwise.}
    \end{array}
    \right.
  \end{equation*}
  In particular, $(\phi_i(X^i_t),\phi_j(X^j_t))$ cannot converge to $(0,0)$ during the interval of time $t\in [t_{2n+1},t_{2n+2}[$ and thus $T^{ij}_0\notin[t_{2n+1},t_{2n+2}[$ almost surely, for any $n\geq 0$. We deduce that
\begin{equation*}
    P\left(T_0^{ij}<+\infty\right)\leq \sum_{n=0}^{+\infty} P\left(T_0^{ij}\in[t_{2n},t_{2n+1}[\right)+P\left(T_0^{ij}\geq \lim_{n\rightarrow\infty}t_n\right).
  \end{equation*}
Our first step consists in proving that the sequence $(t_n)_{n\geq 0}$ converges to $+\infty$ almost surely. In a second step, we prove that $P\left(T_0^{ij}\in[t_{2n},t_{2n+1}[\right)=0$ for each $n\geq 0$.

\me \textit{Step 1: $(t_n)_{n\geq 0}$ converges to $+\infty$ almost surely.}\\
 Since $t_n$ is non-decreasing, it is clear that
\begin{align*}
\left\lbrace (t_n)_{n\geq 0}\text{ does not converge to }+\infty \right\rbrace&\subset \left\lbrace \sup_{n\geq 0} t_n <+\infty \right\rbrace
\\
&\subset\left\lbrace t_{2n+1}-t_{2n}\xrightarrow[n\rightarrow\infty]{} 0 \right\rbrace.
\end{align*}
Using the same It\^o's decomposition as in the first step of Subsection~\ref{subsection:attainability}, we obtain that, up to a negligible set,
\begin{equation*}
\left\lbrace t_{2n+1}-t_{2n}\xrightarrow[n\rightarrow\infty]{} 0 \right\rbrace\subset \left\lbrace\liminf_{n\rightarrow\infty}\gamma\circ\phi_k(X^k_{t_{2n+1}})-\gamma\circ\phi_k(X^k_{t_{2n}})\leq 0\right\rbrace\ \text{for }k=i,j.
\end{equation*}
Since $\gamma$ is non-decreasing, we deduce that, up to a negligible set,
\begin{equation*}
\left\lbrace t_{2n+1}-t_{2n}\xrightarrow[n\rightarrow\infty]{} 0 \right\rbrace\subset \left\lbrace\liminf_{n\rightarrow\infty}\phi_k(X^k_{t_{2n+1}})-\phi_k(X^k_{t_{2n}})\leq 0\right\rbrace\ \text{for }k=i,j.
\end{equation*}
We thus have
\begin{multline*}
\left\lbrace t_{2n+1}-t_{2n}\xrightarrow[n\rightarrow\infty]{} 0 \right\rbrace\\
\subset \left\lbrace\liminf_{n\rightarrow\infty} 
\sqrt{\phi_i(X^i_{t_{2n+1}})^2+\phi_j(X^j_{t_{2n+1}})^2}-\sqrt{\phi_i(X^i_{t_{2n}})^2+\phi_j(X^j_{t_{2n}})^2}\leq 0
\right\rbrace.
\end{multline*}
But the right continuity of the particle system, the continuity of $\phi_i$ and $\phi_j$ and the definition of the sequence $(t_{n})$ imply that
\begin{multline*}
\left\lbrace
\sup_{n\geq 1}{t_n}<+\infty
\right\rbrace\\
\subset
\left\lbrace
\sqrt{\phi_i(X^i_{t_{2n+1}})^2+\phi_j(X^j_{t_{2n+1}})^2}-\sqrt{\phi_i(X^i_{t_{2n}})^2+\phi_j(X^j_{t_{2n}})^2}\geq a_0/2,\ \forall n\geq 1
\right\rbrace.
\end{multline*}
Finally, we deduce that $\left\lbrace (t_n)_{n\geq 0}\text{ does not converge to }+\infty \right\rbrace$ is included in a negligible set, so that $(t_n)_{n\geq 1}$ converges to $+\infty$ almost surely.

\me \textit{Step 2: we have $P\left(T_0^{ij}\in[t_{2n},t_{2n+1}[\right)=0$ for any $n\geq 0$.}\\
  Fix  $n\geq 0$.
  We define the
  positive semi-martingale $Y^i$  by
  \begin{equation}
    \label{EqThExYdecomp}
    Y^i_t=\left\lbrace
    \begin{array}{l}
      \phi_i(X^i_{t_{2n}+t}) \text{ if } t<t_{2n+1}-t_{2n},\\
      a_0\vee \frac{c_0}{k_0} \text{ if } t\geq t_{2n+1}-t_{2n}.
    \end{array}
    \right.
  \end{equation}
We also define the local martingale $M^i$ and the adapted process $b^i$ by
\begin{equation*}
M^i_t=\sum_{l=1}^{d'_{i}} \int_0^{t\wedge (t_{2n+1}-t_{2n})} \frac{\partial  \phi_{i}}{\partial x_l}(X^{i}_{t_{2n}+s})\left[\sigma_{i}(t_{2n}+s,o^{i}_{t_{2n}+s},X^{i}_{t_{2n}+s}) dB^{i}_{t_{2n}+s}\right]_l,
\end{equation*}
and
\begin{multline*}
b^i_t=
\1_{t\in[0,t_{2n+1}-t_{2n}[}\left[\sum_{l=1}^{d'_{i}}  \frac{\partial \phi_{i}}{\partial x_l}(X^{i}_{t_{2n}+t})\left[\mu_{i}\right]_l(t_{2n}+t,o^{i}_{t_{2n}+t},X^{i}_{t_{2n}+t})\right.\\
\left.+\frac{1}{2}\sum_{l,k=1}^{d'_{i}}\frac{\partial^2 \phi_{i}}{\partial x_l\partial x_k}(X^{i}_{t_{2n}+t})
\left[\sigma_{i}\sigma_{i}^*\right]_{kl}(t,o^{i}_{t_{2n}+t},X^{i}_{t_{2n}+t})\right].
\end{multline*}
With these definitions, we deduce from the It\^o's formula that
\begin{equation}
\label{equation:Yi-decomposition-one}
dY^i_t=dM^i_t+b^i_tdt+Y^i_t-Y^i_{t\minus}.
\end{equation}
We define the adapted processes $\pi^i$ and $\rho^i$ by
  \begin{equation*}
    \pi^i_t=
    \begin{cases}
      f_i({t_{2n}+t},o^i_{t_{2n}+t},X^i_{t_{2n}+t}),&\text{if } t<t_{2n+1}-t_{2n},\\
      c_0,& \text{if } t\geq t_{2n+1}-t_{2n}
    \end{cases}
    \end{equation*}
    and
    \begin{equation*}
    \rho^i_t=
    \begin{cases}
      g_i({t_{2n}+t},o^i_{t_{2n}+t},X^i_{t_{2n}+t}),&\text{if } t<t_{2n+1}-t_{2n},\\
      -c_0,&\text{if } t\geq t_{2n+1}-t_{2n},
    \end{cases}
  \end{equation*}
  where $f_i$ and $g_i$ are given by Hypothesis \ref{hypothesis:particle-dynamics}.
An immediate computation leads us to
  \begin{equation}
  \label{equation:Mi-square-brackets-decomposition}
    d\langle M^i \rangle_t=(\pi^i_t+\rho^i_t)dt,\ \forall t\geq 0.
  \end{equation}

\me The process $\pi^i$ is a semi-martingale.     We deduce from the It\^o's formula and from the regularity of $f_i$ over $[0,+\infty[\times E_i\times D^{a_0}_i$ that there exist a local
  martingale $N^i$ and a finite variational process $L^i$ such that, for all $t\geq 0$,
   \begin{equation}
   \label{equation:PIi-decomposition}
    d\pi^i_t=dN^i_t+ dL^i_t + \pi^i_t-\pi^i_{t\minus}.
  \end{equation}
   We define the adapted process $\xi^i$ by
   \begin{align*}
     \xi^i_t=\1_{t\in[0,t_{2n+1}-t_{2n}[}\left(\sum_{k=1,l}^{d_i}\frac{\partial f_i}{\partial e_k}\frac{\partial f_i}{\partial e_l}
                                            [s_i s_i^*]_{kl}
          +\sum_{k=1,l}^{d'_i}\frac{\partial f_i}{\partial x_k}\frac{\partial f_i}{\partial x_l}
                                            [\sigma_i \sigma_i^*]_{kl}\right),
   \end{align*}
   where each function involved is implicitly evaluated at $({t_{2n}+t},o^i_{t_{2n}+t},X^i_{t_{2n}+t})$.
We thus obtain
      \begin{equation}
      \label{equation:Ni-decomposition}
     \langle N^i\rangle_t= \xi^i_t dt.
   \end{equation}
Similarly, we define the processes $Y^j$, $M^j$, $b^j$, $\pi^j$, $\rho^j$, $N^j$, $L^j$ and $\xi^j$.

\me By Hypothesis~\ref{hypothesis:jump-measure} and by the definition of $Y^i$ and $Y^j$, it is clear that
\begin{equation}
\label{equation:properties-set-1}
\frac{Y^i_{t}}{\sqrt{\pi^i_t}}-\frac{Y^i_{t\minus}}{\sqrt{\pi^i_{t\minus}}}\geq 0\text{ and }\frac{Y^j_{t}}{\sqrt{\pi^j_t}}-\frac{Y^j_{t\minus}}{\sqrt{\pi^j_{t\minus}}}\geq 0.
\end{equation}
Moreover, the uniform bound assumptions and the regularity assumptions of Hypothesis~\ref{hypothesis:particle-dynamics} imply that the processes $b^i$, $b^j$, $\pi^i$, $\pi^j$, $\rho^i$, $\rho^j$, $\xi^i$ and $\xi^j$ are uniformly bounded. In particular, there exist some positive constants $b_{\infty}$, $C_{\pi}$ and $C_{\xi}$ such that, for any $t\geq 0$,
\begin{equation}
\label{equation:properties-set-2}
\left\lbrace
\begin{aligned}
b^i_t&\geq -b_{\infty}\text{ and }b^j_t\geq -b_{\infty},\\
\pi^i_t&\leq \pi^i_t+|\rho^i_t|\leq C_{\pi}\text{ and }\pi^j_t\leq \pi^j_t+|\rho^j_t|\leq C_{\pi},\\
\xi^i_t&\leq C_{\xi}\text{ and }\xi^j_t\leq C_{\xi}.
\end{aligned}
\right.
\end{equation}
Setting $c_{\pi}=c_0$ and $k_0=k_g\vee c_0 /a_0$, we deduce from the fourth point of Hypothesis~\ref{hypothesis:particle-dynamics} and the definition of $\pi^i$, $\pi^j$, $\rho^i$ and $\rho^j$ that, for any $t\geq 0$,
\begin{equation}
\label{equation:properties-set-3}
\left\lbrace
\begin{aligned}
c_{\pi}&\leq \pi^i_t\text{ and }c_{\pi}\leq \pi^j_t,\\
|\rho^i_t|&\leq k_0 Y^i_t\text{ and }|\rho^j_t|\leq k_0 Y^j_t.
\end{aligned}
\right.
\end{equation}
By the independence of the particles between the rebirths, we also deduce that, for any $t\geq 0$,
\begin{equation}
\label{equation:properties-set-4}
\langle M^i,M^j\rangle_t=0.
\end{equation}

\bi    We claim now  that the decompositions~\eqref{equation:Yi-decomposition-one}, \eqref{equation:Mi-square-brackets-decomposition}, \eqref{equation:PIi-decomposition} and \eqref{equation:Ni-decomposition} of $Y^i$, $Y^j$, $\langle M^i\rangle$, $\langle M^j\rangle$, $\pi^i$,
    $\pi^j$, $\langle N^i\rangle$ and $\langle N^j\rangle$,  together with the properties \eqref{equation:properties-set-1},
    \eqref{equation:properties-set-2}, \eqref{equation:properties-set-3} and 
    \eqref{equation:properties-set-4}, imply that $(Y^1,Y^2)$ never converges to
    $(0,0)$ almost surely. This is proved in the next section, where a general 
    criterion for non-attainability of $(0,0)$ for non-negative semi-martingales with positive jumps is
    stated (see Proposition~\ref{PrHitting}).  
    
    \me Finally,
    we deduce that $T^{ij}_0\notin [t_{2n},t_{2n+1}[$ almost surely,
        for all $n\geq 0$, concluding the proof of Theorem~\ref{ThExistence}.

 \end{proof}
 
 \section{Non-attainability of $(0,0)$ for semi-martingales}
\label{SeAttainability}
\label{section:non-attainability}

Fix $T>0$ and let $(Y^i_t)_{t\in[0,T]}$, $i=1,2$, be two uniformly bounded non-negative
one-dimensional semi-martingales. This means that there exists a constant $y_{\infty}>0$ such that 
\begin{equation}
0\leq Y^i_t \leq y_{\infty},\ \forall t\geq 0\text{ and }i\in\{1,2\}\text{  almost surely}
\end{equation}
and that
 there exist a continuous local martingale $(M^i_t)_{t\in[0,T]}$ and a continuous finite variational process $I$ such that
\begin{equation}
dY^i_t=dM^i_t+dI^i_t+Y^i_t-Y^i_{t\minus},\ \forall t\geq 0\text{ and }i\in\{1,2\}\text{  almost surely}.
\end{equation}
In this section, we give a sufficient condition for $(Y^1,Y^2)$ not to converge to $(0,0)$ up to
time $T$. More precisely, we set
\begin{equation*}
T_0=\inf\left\lbrace t\in[0,+\infty[,\ \exists s_n\xrightarrow[n\rightarrow\infty]{}t\text{ such that }\lim_{n\rightarrow\infty}(Y^1_{s_n},Y^2_{s_n})=(0,0) \right\rbrace,
\end{equation*}
where $(s_n)$ runs over the set of non-decreasing sequences, and give a sufficient criterion for 
$
T_0=+\infty\text{ almost surely.}
$

\me
 Criteria for non-attainability of $(0,0)$ already exist for time homogeneous stochastic
differential equations (we refer the reader to the papers of Friedman~\cite{Friedman1974},
 Ramasubramanian~\cite{Ramasubramanian1988} and  Delarue~\cite{Delarue2008}). Our result is a generalization, since we do not assume at all that $Y^i$ is given by a stochastic differential equation, nor that it has any time-homogeneity property. In particular, it is not required for $Y^i$ to be a Markov process, which is of first importance in view of our main application (the processes $\phi_i(X^i_t)$ in Section~3 do not fulfil the Markov property). Our proofs are inspired by the recent work of
Delarue~\cite{Delarue2008}, who obtains lower and higher bound for
the hitting time of a corner for a diffusion driven by a time
homogeneous stochastic differential equations reflected in the square.

\me Here is our main assumption.
\begin{hypothesis}
  \label{HyPrHitting}
  For each $i=1,2$, there exist an adapted process $b^i$ and a non-decreasing adapted process $K^i_t$ such that $$dI^i_t=b^i_t dt +dK^i_t.$$ Moreover,  
  there exists a non-negative semi-martingale
  $\pi^i$, whose decomposition is
  \begin{equation*}
    d\pi^i_t=d{N^i_t}+dL^i_t+\pi^i_t-\pi^i_{t\minus},
  \end{equation*}
  where $N^i$ is a continuous local martingale and $L^i$ is a
  continuous finite variational adapted process, and there exist two adapted
  processes $\rho^i_t$ and $\xi^i_t$ and some positive constants
  $b_{\infty},k_0,c_{\pi},C_{\pi}, C_{\xi}$ such that, almost surely,
  \begin{enumerate}
  \item $d\left\langle M^i\right\rangle_t=(\pi^i_{t}+\rho^i_{t}) dt$ and $d\left\langle N^i\right\rangle_t=\xi^i_{t} dt$,
  \item $c_{\pi}\leq \pi^i_t+\rho^i_t\leq
    C_{\pi}$, $|\rho^i_t|\leq k_0 Y^i_t$, $\xi_t \leq C_{\xi}$ and $b^i_t\geq -b_{\infty}$ for
    all $t\in[0,T]$
  \item $\left\langle M^1,M^2\right\rangle$  is a non-increasing process.
  \item we assume that, for any time $t\geq 0$,
    \begin{equation*}
     \frac{Y^i_t}{\sqrt{\pi^i_t}}-\frac{Y^i_{t\minus}}{\sqrt{\pi^i_{t\minus}}}\geq 0.
   \end{equation*}
  \end{enumerate}
\end{hypothesis}

\me\textit{Remark 10.} The third point of Hypothesis \ref{HyPrHitting} has the following
geometrical interpretation: when an increment of $M^1$ is non-positive
(that is when $M^1$ goes closer to $0$), the increment of $M^2$ is
non-negative (so that $M^2$ goes farther from $0$), as a consequence
$(M^1,M^2)$ remains away from $0$. A nice graphic representation of
a very similar phenomenon is given by Delarue's \cite[Figure 1]{Delarue2008}.

\me We define the adapted process $\Phi$ by
\begin{equation}
  \label{EqThAttainabilityPhitDef}
  \Phi_t\stackrel{def}{=}-\frac{1}{2}\log\left(\frac{(Y^1_t)^2}{\pi^1_t}+\frac{(Y^2_t)^2}{\pi^2_t}\right).
\end{equation}
For all $\epsilon>0$, we also
define the stopping time $T_{\epsilon}=\inf\{t\in[0,T],\ \Phi_t\geq
{\epsilon}^{-1}\}$.
Since $\pi^i_t$
is uniformly bounded below by $c_{\pi}$, $\Phi_t$ goes to infinity when $(Y^1_t,Y^2_t)$ goes to $(0,0)$, so that
\begin{equation*}
  T_0=\lim_{\epsilon\rightarrow 0} T_{\epsilon}\ \text{almost surely.}
\end{equation*}

\me We are now able to state our non-attainability result.
\begin{proposition}
  \label{PrHitting}
  If Hypothesis~\ref{HyPrHitting} is fulfilled, then $T_0=+\infty$ almost surely. In particular,
  $(Y^1,Y^2)$ doesn't converge to $(0,0)$ in finite time almost surely.
  Moreover, there exists a positive constant $C$ which only
  depends on $b_{\infty},k_0,c_{\pi},C_{\pi}, C_{\xi}$ such that, for
  all $\epsilon^{-1}>\Phi_0$ and any stopping time $T$,
  \begin{equation*}
    P\left(T_{\epsilon}\leq T \right)\leq
    \frac{1}{\epsilon^{-1}-\Phi_0} C\left(E(|L^1|_{T}+|L^2|_{T}+ T)\right),
  \end{equation*}
  where $|L^i|_T$ is the total variation of $L^i$ at time $T$ and
  $\Phi_0$ is defined in \eqref{EqThAttainabilityPhitDef}.
\end{proposition}

\begin{proof}[Proof of Proposition \ref{PrHitting}:]
Since $T_{\epsilon}$ converges to $T_0$ when $\epsilon\rightarrow 0$, if the last part of Proposition~\ref{PrHitting} is fulfilled, then $P(T_0\leq T)=0$ for any deterministic time $T$ and thus  $T_0=+\infty$ almost surely. As a consequence, we only have to prove the second part of the proposition, for a given stopping time $T$ and a fixed value of $\epsilon>0$.

\me The proof is divided into several steps and is organised as follows. We assume that $\langle M^1,M^2\rangle=0$ in the three first steps and consider the general case in the last one. In Step~1, we compute in detail the It\^o's decomposition of the semi-martingale $\Phi$. In Step~2, we introduce the function
 $F:\mathds{R}\mapsto \mathds{R}$ defined by
  \begin{equation*}
    F(r)=\int_0^r \exp\left( C_F e^{-s} \right)ds,
  \end{equation*}
  where $C_F>0$ is a positive constant that will be fixed later in the
  proof. Using It\^o's formula, we prove that there exists a local martingale $H$ and a positive constant $C>0$ such that
  \begin{equation}
  \label{eq:for-step2-only}
    F(\Phi_t)-F(\Phi_0)\leq H_t 
    + C\left(|L^1|_{t}+|L^2|_{t}+t\right).
  \end{equation}
In Step~3, we conclude the proof of Proposition~\ref{PrHitting} in the particular case $\langle M^1,M^2\rangle=0$. In Step~4, we prove that the result remains true in the general case.

\bi \textit{Step 1. It\^o's decomposition of the semi-martingale $(\Phi_t)_{t\in[0,T_{\epsilon}[}$.}\\
Let us introduce the $C^2$ function $\Gamma$ defined by
  \begin{equation*}
    \begin{array}{lcl}
      \Gamma:& \mathds{R}_+^*\times \mathds{R}_+^* \times
      \mathds{R}_+ \times \mathds{R}_+&\rightarrow \mathds{R}\\
       &(\alpha_1,\alpha_2,x_1,x_2)&\mapsto\displaystyle
    -\frac{1}{2}\log\left(\frac{x_1^2}{\alpha_1}+\frac{x_2^2}{\alpha_2}\right),
     \end{array}
  \end{equation*}
so that $\Phi_t=\Gamma(\pi^1_t,\pi^2_t,Y^1_t,Y^2_t)$. The successive derivatives of the function $\Gamma$ are
  \begin{equation*}
  \begin{split}
    \frac{\partial \Gamma}{\partial x_i}&=-\alpha_i^{-1}x_i e^{2\Gamma},\ 
    \frac{\partial^2 \Gamma}{\partial x_i^2}= -\alpha_i^{-1}e^{2\Gamma} + 2 \alpha_i^{-2}x_i^2 e^{4\Gamma},\\
    \frac{\partial \Gamma}{\partial \alpha_i}&= \frac{1}{2}\alpha_i^{-2}x_i^2 e^{2\Gamma},\ 
    \frac{\partial^2 \Gamma}{\partial \alpha_i^2}= -\alpha_i^{-3}x_i^2 e^{2\Gamma}+\frac{1}{2}\alpha_i^{-4}x_i^4 e^{4\Gamma},\\
    \frac{\partial^2 \Gamma}{\partial x_i \partial\alpha_i}&=\alpha_i^{-2}x_i e^{2\Gamma}-\alpha_i^{-3}x_i^3 e^{4\Gamma},\ 
    \frac{\partial^2 \Gamma}{\partial x_i \partial \alpha_j}=-\alpha_i^{-1}\alpha_j^{-2}x_i x_j^2 e^{4\Gamma},\\
    \frac{\partial^2\Gamma}{\partial x_i\partial x_j}&=2\alpha_i^{-1}\alpha_j^{-1}x_i x_j e^{4\Gamma},\ 
    \frac{\partial^2 \Gamma}{\partial \alpha_i\partial \alpha_j}=\frac{1}{2}\alpha_i^{-2}\alpha_j^{-2}x_i^2 x_j^2 e^{4\Gamma}.
    \end{split}
  \end{equation*}
  In particular, we have
  \begin{equation*}
  \begin{split}
  \sum_{i=1,2}\frac{\partial^2 \Gamma}{\partial x_i^2}(\alpha_1,\alpha_2,x_1,x_2)\alpha_i&=\sum_{i=1,2}-\alpha_i^{-1}e^{2\Gamma}\alpha_i + 2 \alpha_i^{-2}x_i^2 e^{4\Gamma}\alpha_i\\
  &=-2e^{2\Gamma}+2e^{4\Gamma}\left(\frac{x_1^2}{\alpha_1}+\frac{x_2^2}{\alpha_2}\right)\\
  &=-2e^{2\Gamma}+2e^{4\Gamma}e^{-2\Gamma}=0,
  \end{split}
  \end{equation*}
  so that, for any $t\in [0,T_{\epsilon}[$,
  \begin{equation*}
    \sum_{i=1,2}\frac{\partial^2 \Gamma}{\partial x_i^2}(\pi^1_{t},\pi^2_{t},Y^1_{t},Y^2_{t})\pi^i_{t}=0,\text{ almost surely.}
  \end{equation*}
    Using the previous equalities and the It\^o's formula, we get
  \begin{equation}
    \label{EqPrHittingDphi}
    \begin{split}
      d\Phi_t= &-\sum_{i=1,2} \frac{Y^i_{t}}{\pi^i_{t}}
      e^{2\Phi_{t}}dM^i_t
      +\sum_{i=1,2}\frac{(Y^i_{t})^2}{2(\pi^i_{t})^{2}}
      e^{2\Phi_{t}} d{N^i_t}
      -\sum_{i=1,2} \frac{Y^i_{t}}{\pi^i_{t}} e^{2\Phi_{t}} dK^i_t\\
      &- \sum_{i=1,2} \frac{Y^i_{t}}{\pi^i_{t}}
      e^{2\Phi_{t}} b^i_t dt
      + \sum_{i=1,2}\frac{(Y^i_{t})^2}{2(\pi^i_{t})^{2}} e^{2\Phi_{t}}dL^i_t\\
      &+\frac{1}{2}\sum_{i=1,2}\left(-\frac{1}{\pi^i_{t}}e^{2\Phi_{t}} 
        + 2 \frac{(Y^i_{t})^2}{(\pi^i_{t})^{2}} e^{4\Phi_{t}}\right)\rho^i_{t} dt
        +\frac{1}{2}\sum_{i\neq
        j\in\{1,2\}}\frac{2Y^i_t Y^j_t}{\pi^i_t\pi^j_t}e^{4\Phi_t}d\langle M^i,M^j\rangle_t\\
      &+\frac{1}{2}\sum_{i=1,2}\left(-\frac{(Y^i_{t})^2}{(\pi^i_{t})^{3}} e^{2\Phi_{t}}
        +\frac{(Y^i_{t})^4}{2(\pi^i_{t})^{4}} e^{4\Phi_{t}}\right)d\left\langle N^i\right\rangle_t      +\frac{1}{2}\sum_{i\neq
        j\in\{1,2\}}\frac{(Y^i_t)^2(Y^j_t)^2}{2(\pi^i_t)^2(\pi^j_t)^2}e^{4\Phi_t}d\langle N^i,N^j\rangle_t\\
      &+\frac{1}{2}\sum_{i=1,2}\left(\frac{Y^i_{t}}{(\pi^i_{t})^2} e^{2\Phi_{t}}
        -\frac{(Y^i_{t})^3}{(\pi^i_{t})^{3}} e^{4\Phi_{t}}\right)d\left\langle M^i,N^i\right\rangle_t
        -\frac{1}{2}\sum_{i\neq
        j\in\{1,2\}}\frac{Y^i_{t}(Y^j_{t})^2}{\pi^i_{t}(\pi^j_{t})^2}
      e^{4\Phi_{t}}d\left\langle
        M^i,N^j\right\rangle_t\\
        &+\Phi_t-\Phi_{t\minus}
    \end{split}
  \end{equation}
  and
  \begin{equation*}
    \begin{split}
      d\left\langle \Phi\right\rangle_t=&\sum_{i=1,2}
      \frac{(Y^i_{t})^2}{(\pi^i_{t})^2} e^{4\Phi_{t}}(\rho^i_{t}+\pi^i_{t})dt
      +\sum_{i=1,2}\frac{(Y^i_{t})^4}{4(\pi^i_{t})^{4}} e^{4\Phi_{t}} d\left\langle N^i\right\rangle_t\\
      & \sum_{i\neq j\in\{1,2\}}\frac{Y^i_t Y^j_t}{\pi^i_t\pi^j_t}e^{4\Phi_t}d\langle M^i,M^j\rangle_t+ \sum_{i\neq j\in\{1,2\}} \frac{(Y^i_t)^2(Y^j_t)^2}{4(\pi^i_t)^2(\pi^j_t)^2}e^{4\Phi_t}d\langle N^i,N^j\rangle_t \\
      &-\sum_{i=1,2} \frac{(Y^i_{t})^3}{2(\pi^i_{t})^{3}} e^{4\Phi_{t}}d\left\langle M^i,N^i\right\rangle_t
      - \sum_{i\neq j\in\{1,2\}} \frac{Y^i_{t}(Y^j_{t})^2}{2\pi^i_{t}(\pi^j_{t})^{2}}
      e^{4\Phi_{t}} d\left\langle M^i,N^j\right\rangle_t.
    \end{split}
  \end{equation*}

\bi \textit{Step 2: proof of~\eqref{eq:for-step2-only}.}\\
  Let $C_F>0$ be a positive constant that will be fixed later in the
  proof and define the function $F:\mathds{R}\mapsto \mathds{R}$ by
  \begin{equation*}
    F(r)=\int_0^r \exp\left( C_F e^{-s} \right)ds.
  \end{equation*} 
  Setting $D_F=C_F y_{\infty}\sqrt{2}/\sqrt{c_{\pi}}$, one can easily check that
  \begin{equation*}
    1\leq F'(r)\leq e^{D_F}\text{ and }F''(r)=-C_Fe^{-r}F'(r),\ \forall r\geq -\frac{1}{2}\log\left(2\frac{y_{\infty}^2}{c_{\pi}}\right),
  \end{equation*}
  where $-\frac{1}{2}\log\left(2\frac{y_{\infty}^2}{c_{\pi}}\right)$ is a lower bound for the process $(\Phi_t)$.
  By It\^o's formula, we deduce that, for any $t\in [0,T_{\epsilon}[$,
  \begin{equation}
    \label{EqPrHiDecompF}
    F(\Phi_t)-F(\Phi_0)
    =\int_0^t F'(\Phi_{s})d\Phi^c_s
    -\frac{C_F}{2}\int_0^t e^{-\Phi_{s}}F'(\Phi_{s})d\left\langle \Phi\right\rangle_s
    +\sum_{0\leq s\leq t} F(\Phi_s)-F(\Phi_{s\minus}),
  \end{equation}
  where $d\Phi^c_s$ is the continuous part of $d\Phi_s$.
  Our aim is to prove that
  \begin{equation*}
    F(\Phi_t)-F(\Phi_0)\leq H_t 
    + C\left(|L^1|_{t}+|L^2|_{t}+t\right),
  \end{equation*}
  where $H$ is the local martingale defined for any $t>0$ by
      \begin{equation*}
    H_t=-\sum_{i=1,2} \int_0^{t\wedge T_{\epsilon}} \frac{Y^i_{s}}{\pi^i_{s}}  e^{2\Phi_{s}} F'(\Phi_{s})dM^i_s
      +\sum_{i=1,2}\int_0^{t\wedge T_{\epsilon}}\frac{(Y^i_{s})^2}{2(\pi^i_{s})^{2}} e^{2\Phi_{s}}  F'(\Phi_{s}) d{N^i_s}.
  \end{equation*}
  This is done below using~\eqref{EqPrHittingDphi} and proving lower or higher bounds for each term on the right hand side of~\eqref{EqPrHiDecompF}, respectively in step 2a, in step 2b and in step 2c.

\bi \textit{Step 2a.} Let us prove that there exists
a positive constant $C'>0$, which does not depend on $C_F$, such that
  \begin{equation}
    \label{EqPrHiBorne1}
    \int_0^t F'(\Phi_{s})d\Phi^c_s
    \leq H_t+ C' \int_0^t e^{\Phi_s}F'(\Phi_s) ds
    +\frac{e^{D_F}}{2c_{\pi}}\left(|L^1|_t+|L^2|_t\right)
    +\frac{3e^{D_F}C_{\xi}t}{2c_{\pi}^2}.
  \end{equation}
Since $K^i$ is non-decreasing, we have
  \begin{equation*}
    -\sum_{i=1,2} \int_0^t \frac{Y^i_{s}}{\pi^i_{s}} e^{2\Phi_{s}}F'(\Phi_{s}) dK^i_s\leq 0.
  \end{equation*}
  One can easily check that $ Y^i_t
  e^{\Phi_t}\leq \sqrt{\pi^i_t}$, then
$    \frac{Y^i_t}{\pi^i_t}e^{\Phi_t}\leq \frac{1}{\sqrt{c_{\pi}}}.$
  Since $b^i_t\geq -b_{\infty}$, we have 
  \begin{equation*}
    - \sum_{i=1,2} \int_0^t \frac{Y^i_{s}}{\pi^i_{s}} F'(\Phi_{s}) e^{2\Phi_{s}} b^i_s ds
    \leq \frac{2 b_{\infty}}{\sqrt{c_{\pi}}}\int_0^t  e^{\Phi_s} F'(\Phi_s)ds.
  \end{equation*}
  Since the derivative $F'$ takes its values in $[1,e^{D_F}]$ and since 
  $\frac{(Y^i_t)^2}{(\pi^i_t)^2}e^{2\Phi_t}\leq \frac{1}{c_{\pi}}$, we deduce that
  \begin{equation*}
     \sum_{i=1,2} \int_0^t \frac{(Y^i_{s})^2}{2(\pi^i_{s})^{2}}  e^{2\Phi_{s}} F'(\Phi_{s}) dL^i_s
     \leq \frac{e^{C_F}}{2c_{\pi}}\left(|L^1|_t+|L^2|_t\right).
  \end{equation*}
  By the second point of Hypothesis~\ref{HyPrHitting}, we have
$    |\rho^i_t|e^{\Phi_t}\leq k_0 Y^i_t e^{\Phi_t} \leq k_0 \sqrt{\pi^i_t} \leq k_0 \sqrt{C_{\pi}}$. As a consequence, using that $\frac{(Y^i_t)^2}{(\pi^i_t)^2}e^{2\Phi_t}\leq \frac{1}{c_{\pi}}$ and $\pi^i_t\geq c_{\pi}$, we deduce that
  \begin{equation*}
      \frac{1}{2}\sum_{i=1,2}\int_0^t\left(-\frac{1}{\pi^i_{s}}e^{2\Phi_{s}}
        + 2 \frac{(Y^i_{s})^2}{(\pi^i_{s})^{2}}
        e^{4\Phi_{s}}\right) \rho^i_{s} F'(\Phi_s) ds
      \leq \frac{3k_0\sqrt{C_{\pi}}}{c_{\pi}}\int_0^t e^{\Phi_s}  F'(\Phi_s)ds.
  \end{equation*}
  By the Kunita-Watanabe inequality (see \cite[Corollary 1.16 of
  Chapter IV]{Revuz1999}) and by Hypothesis~\ref{HyPrHitting} (first and second points), we get, for all predictable process $h_s$ and any couple $i,j\in\{1,2\}$,
  \begin{equation*}
    \left|\int_0^t h_s 
      \left\langle  M^i,N^j \right\rangle_s\right|
    \leq \sqrt{\int_0^t |h_s|
      \left\langle  M^i \right\rangle_s}\sqrt{\int_0^t |h_s|
      \left\langle  N^j \right\rangle_s}
    \leq \sqrt{ C_{\pi} C_{\xi}}\int_0^t |h_s| ds,
  \end{equation*}
and, similarly,
\begin{equation*}
    \left|\int_0^t h_s 
      \left\langle  M^i,M^j \right\rangle_s\right|
    \leq  C_{\pi}\int_0^t |h_s| ds\text{ and }
        \left|\int_0^t h_s 
      \left\langle  N^i,N^j \right\rangle_s\right|
    \leq  C_{\xi}\int_0^t |h_s| ds,
\end{equation*}
In particular, for any $t\in[0,T_{\epsilon}[$, using $\frac{(Y^i_s)^2}{(\pi^i_s)^2}e^{2\Phi_s}\leq \frac{1}{c_{\pi}}$ and $\pi^i_s\geq c_{\pi}$, we deduce that
  \begin{equation*}
    \frac{1}{2}\sum_{i=1,2}\int_0^t \left(-\frac{(Y^i_{s})^2}{(\pi^i_{s})^{3}} e^{2\Phi_{s}}
      +\frac{(Y^i_{s})^4}{2(\pi^i_{s})^{4}} 
      e^{4\Phi_{s}}\right)F'(\Phi_s)d\left\langle N^i\right\rangle_s
    \leq \frac{e^{D_F}C_{\xi}t}{c_{\pi}^2}
  \end{equation*}
  and
  \begin{equation*}
  \frac{1}{2}\sum_{i\neq
        j\in\{1,2\}}\int_0^t \frac{(Y^i_s)^2(Y^j_s)^2}{2(\pi^i_s)^2(\pi^j_s)^2}e^{4\Phi_s}F'(\Phi_s)d\langle N^i,N^j\rangle_s
        \leq \frac{e^{D_F}C_{\xi}t}{2c_{\pi}^2}.
  \end{equation*}
  We also deduce that
  \begin{equation*}
    \frac{1}{2}\sum_{i=1,2}\int_0^t \left(\frac{Y^i_{s}}{(\pi^i_{s})^2} e^{2\Phi_{s}}
      -\frac{(Y^i_{s})^3}{(\pi^i_{s})^{3}} e^{4\Phi_{s}}\right)
    F'(\Phi_s)d\left\langle M^i,N^i\right\rangle_s
    \leq \frac{2\sqrt{C_{\pi}C_{\xi}}}{c_{\pi}^{3/2}}\int_0^t e^{\Phi_s}F'(\Phi_s)ds
  \end{equation*}
  and 
  \begin{equation*}
    -\frac{1}{2}\sum_{i\neq
        j\in\{1,2\}}\int_0^t \frac{Y^i_{s}(Y^j_{s})^2}{\pi^i_{s}(\pi^j_{s})^2}
      e^{4\Phi_{s}}F'(\Phi_s)d\left\langle
        M^i,N^j\right\rangle_s
      \leq \frac{\sqrt{C_{\pi}C_{\xi}}}{c_{\pi}^{3/2}}\int_0^t e^{\Phi_s} F'(\Phi_s) ds.
  \end{equation*}
  Setting $C'= \frac{2 b_{\infty}}{\sqrt{c_{\pi}}} +
      \frac{3k_0\sqrt{C_{\pi}}}{c_{\pi}} + \frac{3\sqrt{
          C_{\pi} C_{\xi}}}{c_{\pi}^{3/2}}>0$
          and using equation~\eqref{EqPrHittingDphi}
(recall that we assumed $\langle M^1,M^2 \rangle=0$),
           we deduce that the higher bound~\eqref{EqPrHiBorne1} 
holds almost surely, for any $t\in[0,T_{\epsilon}[$.

\bi \textit{Step 2b.}  We prove now the following lower bound for $\int_0^t
  e^{-\Phi_{s}}F'(\Phi_{s})d\left\langle \Phi\right\rangle_s$,
    \begin{equation}
    \label{EqPrHiBorne2}
    \int_0^t e^{-\Phi_{s}}F'(\Phi_{s})d\left\langle \Phi\right\rangle_s
    \geq \frac{c_{\pi}}{C_{\pi}} \int_0^t e^{\Phi_s} F'(\Phi_s)ds 
    -\left(\frac{C_{\xi}}{2c_{\pi}^2}
   +\frac{k_0\sqrt{C_{\pi}}}{c_{\pi}}+\frac{2\sqrt{C_{\pi}C_{\xi}}}{c_{\pi}^{3/2}}\right)e^{D_F}t.
  \end{equation}
  One can easily check (using the second point of Hypothesis~\ref{HyPrHitting}) that
  \begin{equation*}
    \frac{e^{2\Phi_s}}{C_{\pi}}\leq
    \sum_{i=1,2} \frac{(Y^i_{s})^2}{(\pi^i_{s})^2}e^{4\Phi_{s}}
    \leq \frac{e^{2\Phi_s}}{c_{\pi}}
    \text{ and }
    \rho^i_s\geq -k_0 Y^i_s \geq -k_0 \sqrt{C_{\pi}}e^{-\Phi_s}.
  \end{equation*}
Using that $\pi^i_s\geq c_\pi$, we deduce that
  \begin{equation*}
    \sum_{i=1,2}\int_0^t   \frac{(Y^i_{s})^2}{(\pi^i_{s})^2} e^{4\Phi_{s}}(\pi^i_{s}+\rho^i_{s})
    e^{-\Phi_s}F'(\Phi_s)ds
    \geq \frac{c_{\pi}}{C_{\pi}} \int_0^t e^{\Phi_s} F'(\Phi_s)ds -\frac{k_0\sqrt{C_{\pi}}}{c_{\pi}} e^{D_F} t.
  \end{equation*}
  The process $\langle N^i\rangle$ being non-decreasing and $F'$ being positive, we have
  \begin{equation*}
    \sum_{i=1,2}\int_0^t \frac{(Y^i_{s})^4}{4(\pi^i_{s})^{4}}
    e^{4\Phi_{s}} F'(\Phi_s)e^{-\Phi_s}d\left\langle N^i\right\rangle_s\geq 0.
  \end{equation*}
  The same application of the Kunita-Watanabe inequality as above leads us to
   \begin{equation*}
   \begin{split} 
   \sum_{i\neq j\in\{1,2\}} \int_0^t \frac{(Y^i_s)^2(Y^j_s)^2}{4(\pi^i_s)^2(\pi^j_s)^2}e^{4\Phi_s}F'(\Phi_s)e^{-\Phi_s}d\langle N^i,N^j\rangle_s&\geq - \frac{C_{\xi}e^{D_F}t}{2c_{\pi}^2}\int_0^t e^{-\Phi_s} ds\\
   & \geq - \frac{C_{\xi}}{2c_{\pi}^2}e^{D_F}t,
   \end{split}
  \end{equation*}
  since the definition of $\Phi$ implies that it is uniformly bounded below by $0$.
We also have  
  \begin{equation*}
  -\sum_{i=1,2}\int_0^t \frac{(Y^i_{s})^3}{2(\pi^i_{s})^{3}} e^{4\Phi_{s}}
  F'(\Phi_s)e^{-\Phi_s}d\left\langle M^i,N^i\right\rangle_s
  \geq -\frac{\sqrt{C_{\pi}C_{\xi}}}{c_{\pi}^{3/2}} e^{D_F} t
  \end{equation*}
  and
  \begin{equation*}
    -\sum_{i\neq j\in\{1,2\}} \int_{0}^t \frac{Y^i_{s}(Y^j_{s})^2}
    {2\pi^i_{s}(\pi^j_{s})^{2}} e^{4\Phi_{s}} F'(\Phi_s)e^{-\Phi_s}d\left\langle M^i,N^j\right\rangle_s
    \geq - \frac{\sqrt{C_{\pi}C_{\xi}}}{c_{\pi}^{3/2}} e^{D_F}t.
  \end{equation*}
  We finally deduce by~\eqref{EqPrHittingDphi} (where $\langle M^1,M^2\rangle=0$) that~\eqref{EqPrHiBorne2} holds.

\bi \textit{Step 2c.} The jumps of $\Phi_t$ are non-positive and $F$ is increasing, thus
  \begin{equation}
    \label{EqPrHiBorne3}
     \sum_{0\leq s\leq t} F(\Phi_s)-F(\Phi_{s\minus})\leq 0.
  \end{equation}
  Finally, by \eqref{EqPrHiBorne1}, \eqref{EqPrHiBorne2} and  \eqref{EqPrHiBorne3}, we deduce
  from \eqref{EqPrHiDecompF} that
  \begin{multline*}
    F(\Phi_t)-F(\Phi_0)\leq H_t+ \left(C'-\frac{C_F c_{\pi}}{2C_{\pi}}\right)
    \int_0^t e^{\Phi_s}F'(\Phi_s) ds
    +\frac{e^{D_F}}{2c_{\pi}}\left(|L^1|_t+|L^2|_t\right)\\
    +\left(\frac{(3+C_F)C_{\xi}}{2c_{\pi}^2}
      +\frac{k_0 C_F \sqrt{C_{\pi}}}{2c_{\pi}}
      +\frac{C_F\sqrt{C_{\pi}C_{\xi}}}{c_{\pi}^{3/2}}\right)
    e^{D_F}t.
  \end{multline*}
  Choosing $C_F= 2C_{\pi}C'/c_{\pi}$, we have proved that there exists $C>0$ such that, for any $t\in[0,T_{\epsilon}[$,
  \begin{equation*}
    F(\Phi_t)-F(\Phi_0)\leq H_t 
    + C\left(|L^1|_{t}+|L^2|_{t}+t\right).
  \end{equation*}

\bi \textit{Step 3: conclusion in the case  $\left\langle
    M^1,M^2\right\rangle=0$.}\\
       Let
  $(\theta'_n)_{n\in\mathds{N}}$ be an increasing sequence of stopping times which converge to $+\infty$ such
  that $(H_t)_{t\in[0,\theta'_n]}$ is a martingale for any $n\geq 0$. Let $(\theta''_n)_{n\geq 0}$ be the non-decreasing sequence of stopping times defined by $\theta''_n=\inf\{t\in[0,T],\
  \exists i \text{ st }\int_0^{t} d|L^i|_s\geq n\}\wedge T$; since $L^i$ is a finite variational process, $(\theta''_n)$ converges to $+\infty$ when $n$ goes to $\infty$. 
  Using Step 2, we deduce that
  \begin{equation}
    \label{EqPrHittingAssumed}
    E\left(\int_{\Phi_0}^{\Phi_{ T_{\epsilon}\wedge
          \theta'_{n'}\wedge\theta''_{n''}}} \exp(e^{C_F}e^{-u}) du\right)
    \leq C\,E(|L^1|_{\theta''_{n''}}+|L^2|_{\theta''_{n''}} + T).
  \end{equation}
   Remark that
  $\Phi_{t\wedge T_{\epsilon}\wedge \theta'_{n'}\wedge\theta''_{n''}}$
  reaches $\epsilon^{-1}$ if and only if $T_{\epsilon}\leq
  \theta'_{n'}\wedge\theta''_{n''}$. Using the right continuity of
  $Y^1$, $Y^2$, $\pi^1$ and $\pi^2$, we thus deduce that
  \begin{eqnarray*}
    P\left(T_{\epsilon}\leq \theta'_{n'}\wedge\theta''_{n''}
    \right)&=&P\left(\Phi_{ T_{\epsilon}\wedge
        \theta'_{n'}\wedge\theta''_{n''}}-\Phi_0\geq\epsilon^{-1}-\Phi_0\right)\\
    &\leq&P\left(\int_{\Phi_0}^{\Phi_{ T_{\epsilon}\wedge
          \theta'_{n'}\wedge\theta''_{n''}}} \exp(e^{C_F}e^{-u}) du\geq
      \epsilon^{-1}-\Phi_0\right),
  \end{eqnarray*}
  since $r-q\leq \int_q^r\exp(e^{C_F}e^{-u}) du$ for all $0\leq q\leq
  r$. Finally, using the Markov inequality and
  \eqref{EqPrHittingAssumed}, we get, for all $\epsilon^{-1}> \Phi_0$,
  \begin{equation*}
    P\left(T_{\epsilon}\leq \theta'_{n'}\wedge\theta''_{n''} \right)\leq
    \frac{1}{\epsilon^{-1}-\Phi_0} C\, E\left(|L^1|_{\theta''_{n''}}+|L^2|_{\theta''_{n''}} + T\right).
  \end{equation*}
 Since
  $(\theta'_{n'})$ and $(\theta''_{n''})$ converge to $+\infty$ almost surely,
  letting $n'$ and $n''$ go to $\infty$ implies the last part of Proposition~\ref{PrHitting}, which implies the whole proposition.

\bi \textit{Step 4: conclusion of the proof in the general case.}\\
  Assume now that $\left\langle M^1,M^2\right\rangle$ is non-increasing. We define $\Phi'_t$ as the
  process starting from $\Phi_0$ and whose increments are defined by the
  right term of \eqref{EqPrHittingDphi} but removing the $d\langle M^i,M^j\rangle$ terms. We also define
	$T'_{\epsilon}=\inf\{t\geq 0,\ \Phi'_t\geq \epsilon^{-1}\}$.  
  On the one hand, the same
  calculation as above leads to
  \begin{equation}
    \label{EqPrHittingFphiPrim}
    P(T'_{\epsilon}\leq T)\leq  \frac{1}{\epsilon^{-1}-\Phi_0} C\, E\left(|L^1|_{T}+|L^2|_{T} + T\right).
  \end{equation}
  On the other hand,
  \begin{equation*}
    d\Phi_t=d\Phi'_t 
    + \frac{1}{2}\sum_{i\neq j\in\{1,2\}} \frac{2 Y^i_t Y^j_t}{\pi^i_t \pi^j_t}e^{4\Phi_t} d\langle M^i,M^j\rangle_t.
  \end{equation*}
 Since $\left\langle M^1,M^2\right\rangle$ is assumed to be non-increasing (third point of Hypothesis~\ref{HyPrHitting}), we deduce that $\Phi_t\leq \Phi'_t$. It yields that $T'_{\epsilon}\leq T_{\epsilon}$ almost surely, so that Proposition~\ref{PrHitting} holds even if $\left\langle M^1,M^2\right\rangle\neq 0$.
  
  \end{proof}

\paragraph{Acknowledgements.} This work has been mostly written during my PhD thesis and I am extremely grateful to my PhD advisor Sylvie M\'el\'eard for her substantial help. This work  benefited from the support of the ANR MANEGE (ANR-09-
BLAN-0215), from the Chair "Mod\'elisation Math\'ematique et Biodiversit\'e" of Veolia Environnement-\'Ecole Polytechnique-Museum National d'Histoire Naturelle-Fondation X and from the TOSCA team (INRIA Grand-Est Nancy, France).


\begin{thebibliography}{10}

\bibitem{Ben-Ari2009}
I.~Ben-Ari and R.~G. Pinsky.
\newblock Ergodic behavior of diffusions with random jumps from the boundary.
\newblock {\em Stoch. Proc. Appl.}, 119(3):864 -- 881,
  2009.

\bibitem{Bieniek2009}
M.~Bieniek, K.~Burdzy, and S.~Finch.
\newblock Non-extinction of a Fleming-Viot particle model.
\newblock {\em Probab. Theory Rel.}, pages 1--40, 2011.

\bibitem{Bieniek2011}
M.~{Bieniek}, K.~{Burdzy}, and S.~{Pal}.
\newblock {Extinction of Fleming-Viot-type particle systems with strong drift}.
\newblock {\em ArXiv e-prints}, Oct. 2011.

\bibitem{Burdzy1996}
K.~Burdzy, R.~Holyst, D.~Ingerman, and P.~March.
\newblock Configurational transition in a fleming-viot-type model and
  probabilistic interpretation of laplacian eigenfunctions.
\newblock {\em J. Phys. A}, 29(29):2633--2642, 1996.

\bibitem{Burdzy2000}
K.~Burdzy, R.~Holyst and P.~March.
\newblock A {F}leming-{V}iot particle representation of the {D}irichlet {L}aplacian.
\newblock {\em Comm. Math. Phys.},
214(3):679--703, 200.

\bibitem{DelMoral2003}
P.~Del~Moral and L.~Miclo.
\newblock Particle approximations of {L}yapunov exponents connected to
  {S}chr\"odinger operators and {F}eynman-{K}ac semigroups.
\newblock {\em ESAIM Probab. Stat.}, 7:171--208, 2003.

\bibitem{Delarue2008}
F.~Delarue.
\newblock Hitting time of a corner for a reflected diffusion in the square.
\newblock {\em Ann. Inst. Henri Poincar\'e Probab. Stat.}, 44(5):946--961,
  2008.

\bibitem{Delfour2001}
M.~C. Delfour and J.-P. Zol{\'e}sio.
\newblock {\em Shapes and geometries}, volume~4 of {\em Advances in Design and
  Control}.
\newblock Society for Industrial and Applied Mathematics (SIAM), Philadelphia,
  PA, 2001.
\newblock Analysis, differential calculus, and optimization.

\bibitem{Ferrari2007}
P.~A. Ferrari and N.~Mari{\'c}.
\newblock Quasi stationary distributions and {F}leming-{V}iot processes in
  countable spaces.
\newblock {\em Electron. J. Probab.}, 12:no. 24, 684--702 (electronic), 2007.

\bibitem{Friedman1974}
A.~Friedman.
\newblock Nonattainability of a set by a diffusion process.
\newblock {\em Trans. Amer. Math. Soc.}, 197:245--271, 1974.

\bibitem{Grigorescu2004}
I.~Grigorescu and M.~Kang.
\newblock Hydrodynamic limit for a {F}leming-{V}iot type system.
\newblock {\em Stoch. Proc. Appl.}, 110(1):111--143, 2004.

\bibitem{Grigorescu2007}
I.~Grigorescu and M.~Kang.
\newblock Ergodic properties of multidimensional {B}rownian motion with
  rebirth.
\newblock {\em Electron. J. Probab.}, 12:no. 48, 1299--1322, 2007.

\bibitem{Grigorescu2011}
I.~Grigorescu and M.~Kang.
\newblock Immortal particle for a catalytic branching process.
\newblock {\em Probab. Theory Rel.}, pages 1--29, 2011.
\newblock 10.1007/s00440-011-0347-6.

\bibitem{Kolb2010}
M.~Kolb and D.~Steinsaltz.
\newblock Quasilimiting behavior for one-dimensional diffusions with killing.
\newblock \textit{Ann. Prob.}, 40:1, 162-212, 2012.

\bibitem{Kolb2011a}
M.~{Kolb} and A.~{W{\"u}bker}.
\newblock {On the Spectral Gap of Brownian Motion with Jump Boundary}.
\newblock {\em Electron. J. Probab.}, 16, 1214--1237.

\bibitem{Kolb2011}
M.~{Kolb} and A.~{W{\"u}bker}.
\newblock {Spectral Analysis of Diffusions with Jump Boundary}.
\newblock {\em J. Funct. Anal.}, 261:7, 1992--2012.

\bibitem{Lambert2007}
A.~Lambert.
\newblock Quasi-stationary distributions and the continuous-state branching
  process conditioned to be never extinct.
\newblock {\em Electron. J. Probab.}, 12:no. 14, 420--446, 2007.

\bibitem{Lobus2009}
J.-U. L{\"o}bus.
\newblock A stationary {F}leming-{V}iot type {B}rownian particle system.
\newblock {\em Math. Z.}, 263(3):541--581, 2009.

\bibitem{Meleard2011}
S.~M\'el\'eard and D.~Villemonais
\newblock Quasi-stationary distributions and population processes.
\newblock {\em Probab. Surveys}, 9:340-410, 2012.

\bibitem{Pollett}
P.~Pollett.
\newblock Quasi-stationary distributions : a bibliography.
\newblock http://www.maths.uq .edu.au/$\sim$pkp/papers/qsds/qsds.pdf.

\bibitem{Ramasubramanian1988}
S.~Ramasubramanian.
\newblock Hitting of submanifolds by diffusions.
\newblock {\em Probab. Theory Rel.}, 78(1):149--163, 1988.

\bibitem{Revuz1999}
D.~Revuz and M.~Yor.
\newblock {\em Continuous martingales and {B}rownian motion}, volume 293 of
  {\em Grundlehren der Mathematischen Wissenschaften [Fundamental Principles of
  Mathematical Sciences]}.
\newblock Springer-Verlag, Berlin, third edition, 1999.

\bibitem{Rousset2006}
M.~Rousset.
\newblock On the control of an interacting particle estimation of
  {S}chr\"odinger ground states.
\newblock {\em SIAM J. Math. Anal.}, 38(3):824--844 (electronic), 2006.

\bibitem{Villemonais2010}
D.~Villemonais.
\newblock Interacting particle systems and Yaglom limit approximation of
  diffusions with unbounded drift.
\newblock {\em Electron. J. Probab.}, 16:1663--1692, 2011.

\bibitem{Zhen2010}
W.~Zhen and X.~Hua.
\newblock Multi-dimensional reflected backward stochastic differential
  equations and the comparison theorem.
\newblock {\em Acta Mathematica Scientia}, 30(5):1819 -- 1836, 2010.

\end{thebibliography}

\end{document}